\documentclass{scrartcl}

\usepackage[english]{babel}

\usepackage{amsthm,amsmath,amsfonts,amssymb}
\usepackage[autostyle]{csquotes}
\usepackage{hyperref}
\usepackage{microtype}

\DeclareMathOperator{\rk}{rk}

\newtheorem{corollary}{Corollary}
\newtheorem{example}{Example}
\newtheorem{lemma}{Lemma}
\newtheorem{proposition}{Proposition}
\newtheorem{theorem}{Theorem}

\newcommand{\gaussmnum}[3]{\left[\begin{smallmatrix}{#1}\\{#2}\end{smallmatrix}\right]_{#3}}
\newcommand{\gaussmset}[2]{\left[\begin{smallmatrix}{#1}\\{#2}\end{smallmatrix}\right]}
\newcommand{\F}{\ensuremath{\mathbb{F}}}
\newcommand{\Fq}{\ensuremath{\F_q}}
\newcommand{\Fqn}{\ensuremath{V}}
\newcommand{\PFqn}{\ensuremath{P(\Fqn)}}
\newcommand{\Gqnk}{\ensuremath{\gaussmset{\Fqn}{k}}}
\newcommand{\dham}{\mathrm{d}_{\mathrm{h}}}
\newcommand{\rdist}{\mathrm{d}_{\mathrm{r}}}
\newcommand{\sdist}{\mathrm{d}_{\mathrm{s}}}
\newcommand{\idist}{\mathrm{d}_{\mathrm{i}}}

\title{Asymptotic bounds for the sizes of constant dimension codes and an improved lower bound}
\author{Daniel Heinlein and Sascha Kurz\thanks{The work was supported by the ICT COST Action IC1104
and grants KU 2430/3-1, WA 1666/9-1 -- ``Integer Linear Programming Models for Subspace Codes and
Finite Geometry'' -- from the German Research Foundation.}
\\
University of Bayreuth\\
\href{mailto:Daniel.Heinlein@uni-bayreuth.de}{Daniel.Heinlein@uni-bayreuth.de},
\href{mailto:Sascha.Kurz@uni-bayreuth.de}{Sascha.Kurz@uni-bayreuth.de}
}

\begin{document}

\maketitle

\begin{abstract}
  We study asymptotic lower and upper bounds for the sizes of constant dimension codes
  with respect to the subspace or injection distance, which is used in random linear network coding.
  In this context we review known upper bounds and show relations between them. A slightly improved version
  of the so-called linkage construction
  is presented which is e.g.\ used
  to construct constant dimension codes with subspace distance $d=4$, dimension $k=3$ of the codewords
  for all field sizes $q$, and sufficiently large dimensions $v$ of the ambient space, that exceed
  the MRD bound, for codes containing a lifted MRD code, by Etzion and Silberstein.
\end{abstract}

Keywords: constant dimension codes, subspace distance, injection distance, random network coding

\section{Introduction}

Let $V\cong{\F}_q^v$ be a $v$-dimensional vector space over the finite field {\Fq} with $q$ elements.
By {\Gqnk} we denote the set of all $k$-dimensional subspaces in {\Fqn}, where $0\le k\le v$,
which has size $\gaussmnum{v}{k}{q}:= \prod_{i=1}^{k} \frac{q^{v-k+i}-1}{q^i-1}$. More general,
the set ${\PFqn}$ of all subspaces of $V$ forms a metric space with respect to the subspace distance defined by
$\sdist(U,W) = \dim(U+W)-\dim(U \cap W)=\dim(U)+\dim(W)-2\dim(U \cap W)$, see \cite{MR2451015}, and the injection distance
defined by $\idist(U,W)=\max\{\dim(U),\dim(W)\}-\dim(U \cap W)$, see \cite{silva2009metrics}. Coding Theory on {\PFqn} is motivated by K{\"o}tter
and Kschi\-schang \cite{MR2451015} via error correcting random network coding, see  \cite{MR1768542}. In this context it is
natural to consider codes $\mathcal{C}\subseteq P(V)$ where each codeword, i.e., each element of $\mathcal{C}$, has the same dimension $k$,
called \emph{constant dimension code} (cdc), since this knowledge can be exploited by decoders. For constant dimension
codes we have $\sdist(U,W)=2\idist(U,W)$, so that we will only consider the subspace distance in this paper. By $(v,N,d;k)_q$ we denote a cdc in {\Fqn}
with minimum (subspace) distance $d$ and size $N$, where the dimensions of each codeword is $k\in \{0,1,\ldots,v\}$. As usual,
a cdc $\mathcal{C}$ has the \emph{minimum distance} $d$, if $d \le \sdist(U,W)$ for all $U \ne W \in \mathcal{C}$ and equality is attained at least
once. If $\# \mathcal{C}=1$, we set the minimum distance to $\infty$. The corresponding maximum size is denoted by $A_q(v,d;k)$, where we allow the
minimum distance to be larger than $d$. In \cite{MR2451015} the authors provided lower and
upper bounds for $A_q(v,d;k)$ which are less than a factor of $4$ apart. For increasing field size $q$ this factor tends to $1$.
Here, we tighten the corresponding analysis and end up in a factor of less than $2$ for the binary field $q=2$ and a strictly better factor for
larger values of $q$. With respect to lower bounds, we slightly generalize the so-called linkage construction by
Gluesing-Luerssen, Troha / Morrison \cite{MR3543532,gluesing2015cyclic} and Silberstein, (Horlemann-)Trautmann \cite{silberstein2015error}.
This improvement then gives the best known lower bounds for $A_q(v,d;k)$ for many parameters, cf.~the online
tables \url{http://subspacecodes.uni-bayreuth.de} associated with \cite{HKKW2016Tables}. For codes containing a lifted maximum rank
distance (LMRD) code as a subcode an upper bound on the size has been presented in \cite{MR3015712} for some infinite series of parameters.
Codes larger than this MRD bound are very rare. Based on the improved linkage construction we give an infinite series of such
examples.

In this context we mention the following asymptotic result based on the non-constructive probabilistic method. If the subspace distance
$d$ and the dimension $k$ of the codewords is fixed, then the ratio of the lower and upper bound tends to $1$ as the dimension $v$ of the
ambient space approaches infinity, see \cite[Theorem~4.1]{MR829351}, which is implied by a more general result of Frankl and R\"odl on hypergraphs.
The same result, with an explicit error term, was also obtained in \cite[Theorem~1]{blackburn2012asymptotic}. If $d$ and $v-k$ is fixed we have the
same result due to the orthogonal code. If the parameter $k$ can vary with the dimension $v$, then our asymptotic analysis implies there is still
a gap of almost $2$ between the lower and the upper bound of the code sizes for $d=4$ and $k=\left\lfloor v/2\right\rfloor$, which is the worst case.

The remaining part of the paper is organized as follows. In Section~\ref{sec_preliminaries} we collect the basic facts and definitions
for constant dimension codes. Upper bounds on the achievable codes sizes are reviewed in Section~\ref{sec_bounds}. Here, we partially extend
the current knowledge on the relation between these bounds. While most of them are known around 2008 there are some recent improvements
for the subclass of partial spreads, where $d=2k$, which we summarize in Subsection~\ref{subsec_bounds_partial_spreads}. In Section~\ref{sec_linkage}
we present the mentioned improvement of the linkage construction. Asymptotic bounds for the ratio between lower and upper bounds for code sizes
are studied in Section~\ref{sec_asymptotic}. We continue with the upper bound for constant dimension codes containing a lifted MRD code in
Section~\ref{sec_better_than_MRD_bound}, including some numerical results, before we draw a short conclusion in Section~\ref{sec_conclusion}.

\section{Preliminaries}
\label{sec_preliminaries}
For the remainder of the paper we set $V\cong \mathbb{F}_q^v$, where $q$ is a prime power. By $v$ we denote the dimension of $V$. Using the
language of projective geometry, we will call the $1$-dimensional subspaces of $\mathbb{F}_q^v$ points and the $2$-dimensional subspaces lines.
First, we observe that the $q$-binomial coefficient $\gaussmnum{v}{k}{q}$ indeed gives the cardinality of {\Gqnk}.
To this end, we associate with a subspace $U\in {\Gqnk}$ a unique $k\times v$ matrix $X_U$ in row reduced echelon form (rref) having the property
that $\langle X_U\rangle=U$ and denote the corresponding bijection
\[
  \gaussmset{\mathbb{F}_q^v}{k} \rightarrow \{X_U \in \mathbb{F}_q^{k \times v} | \rk(X_U)=k, \text{$X_U$ is in rref}\}
\]
by $\tau$. An example is given by $X_U=\left(\begin{smallmatrix}1&0&0\\0&1&1\end{smallmatrix}\right)\in \mathbb{F}_2^{2 \times 3}$, where
$U=\tau^{-1}(X_U)\in \gaussmset{\mathbb{F}_2^3}{2}$ is a line that contains the three points $(1,0,0)$, $(1,1,1)$, and $(0,1,1)$.
Counting those matrices gives
\[
  \# {\Gqnk}=\prod_{i=0}^{k-1}\frac{q^v-q^i}{q^k-q^i}=\prod_{i=1}^{k} \frac{q^{v-k+i}-1}{q^i-1}=\gaussmnum{v}{k}{q}
\]
for all integers $0\le k\le v$. Especially, we have $\gaussmnum{v}{v}{q}=\gaussmnum{v}{0}{q}=1$.
Given a non-degenerate bilinear form, we denote
by $U^\perp$ the orthogonal subspace of a subspace $U$, which then has dimension $v-\dim(U)$. Then, we have $\sdist(U,W)=\sdist(U^\perp,W^\perp)$,
so that $\gaussmnum{v}{k}{q}=\gaussmnum{v}{v-k}{q}$. The recurrence relation for the usual binomial coefficients generalize to
$\gaussmnum{v}{k}{q}=q^k\gaussmnum{v-1}{k}{q}+\gaussmnum{v-1}{k-1}{q}$.
In order to remove the restriction $0\le k\le v$, we set
$\gaussmnum{a}{b}{q}=0$ for $a\in\mathbb{N}_{\ge 0}$ and $b\in\mathbb{Z}$, whenever $b<0$ or $a<b$. This extension goes in line with
the interpretation of the number of $b$-dimensional subspaces of $\mathbb{F}_q^a$ and respects the orthogonality relation. In order
to write $\sum_{j=0}^{v-1} q^j=\gaussmnum{v}{1}{q}$ for positive integers $q$ in later formulas, we apply the definition of $\gaussmnum{v}{k}{q}$
also in cases where $q$ is not a prime power and set $\gaussmnum{v}{k}{1}=\binom{v}{k}$ for $q=1$.

Using the bijection
$\tau$ we can express the subspace distance between two $k$-dimensional subspaces $U,W\in\Gqnk$ via the rank of a matrix:
\begin{equation}
  \label{eq_d_s_rk}
  \sdist(U,W)=2\dim(U+W)-\dim(U)-\dim(W)=2\left(\rk\!\left(\begin{smallmatrix}\tau(U)\\ \tau(W)\end{smallmatrix}\right)-k\right).
\end{equation}

Using {\Gqnk} as vertex set, we obtain the so-called Grassmann graph, where two vertices are adjacent iff the corresponding subspaces intersect
in a space of dimension $k-1$. It is well-known that the Grassmann graph is distance regular. The injection distance $\idist(U,W)$ corresponds
to the graph distance in the Grassmann graph. Considered as an association scheme one speaks of the $q$-Johnson scheme.

If $\mathcal{C}\subseteq\Gqnk$ is a cdc with minimum subspace distance $d$, we speak of a $(v,\#\mathcal{C},d;k)$ constant dimension code.
In the special case of $d=2k$ one speaks of so-called partial spreads, i.e., collections of $k$-dimensional subspaces with pairwise
trivial intersection.

Besides the injection and the subspace distance we will also consider the Hamming distance $\dham(u,w)=\# \{i \mid u_i \ne w_i\}$,
for two vectors $u,w \in \mathbb{F}_2^v$, and the rank distance $\rdist(U,W)=\rk(U-W)$, for two matrices $U,W\in\mathbb{F}_q^{m\times n}$.
The latter is indeed a metric, as observed in \cite{gabidulin1985theory}. A subset $\mathcal{C}\subseteq \mathbb{F}_q^{m\times n}$ is
called a rank metric code. If the minimum rank-distance of $\mathcal{C}$ is given by $d_r$, we will also speak of an
$(m \times n , \#\mathcal{C} , d_r )_q$ rank metric code in order to specify its parameters. A rank metric code
$\mathcal{C}\subseteq \mathbb{F}_q^{m\times n}$ is called linear if $\mathcal{C}$ forms a subspace of $\mathbb{F}_q^{m\times n}$,
which implies that $\# \mathcal{C}$ has to be a power of the field size $q$.

\begin{theorem}(see \cite{gabidulin1985theory})
  \label{thm_MRD_size}
  Let $m,n\ge d$ be positive integers, $q$ a prime power, and $\mathcal{C}\subseteq \mathbb{F}_q^{m\times n}$ be a rank metric
  code with minimum rank distance $d$. Then, $\# \mathcal{C}\le q^{\max\{n,m\}\cdot (\min\{n,m\}-d+1)}$.
\end{theorem}
Codes attaining this upper bound are called maximum rank distance (MRD) codes. They exist for all (suitable) choices of parameters, which
remains true if we restrict to linear rank metric codes, see  \cite{gabidulin1985theory}. If $m<d$ or $n<d$, then only
$\# \mathcal{C}=1$ is possible, which can be achieved by a zero matrix and may be summarized to the single upper bound
$\# \mathcal{C}\le \left\lceil q^{\max\{n,m\}\cdot (\min\{n,m\}-d+1)}\right\rceil$.
Using an $m\times m$ identity matrix as a prefix one obtains the so-called lifted MRD codes.

\begin{theorem}\cite[Proposition 4]{silva2008rank}
  For positive integers $k,d,v$ with $k\le v$, $d\le 2\min\{k,v-k\}$, and $d$ even, the size of a lifted MRD code in $\Gqnk$ with
  subspace distance $d$ is given by \[M(q,k,v,d):=q^{\max\{k,v-k\}\cdot(\min\{k,v-k\}-d/2+1)}.\] If $d>2\min\{k,v-k\}$, then we have $M(q,k,v,d):=1$.
\end{theorem}

The Hamming distance can be used to lower bound the subspace distance between two codewords (of the same dimension). To this end
let $p: \{M \in \mathbb{F}_q^{k \times v} | \rk(M)=k, \text{M is in rref}\} \rightarrow \{x \in \mathbb{F}_2^v \mid \sum_{i=1}^v x_i = k\}$
denote  the pivot positions of the matrix in rref. For our example $X_U$ we we have $p(X_U)=(1,1,0)$. Slightly abusing notation we
also write $p(U)$ for subspaces $U\in\Gqnk$ instead of $p(\tau(U))$.

\begin{lemma}\cite[Lemma~2]{MR2589964}
\label{lemma_d_s_d_h}
For two subspaces $U,W \le \mathbb{F}_q^v$, we have $\sdist(U,W) \ge \dham(p(U),p(W))$.
\end{lemma}

\section{Upper Bounds}
\label{sec_bounds}
In this section we review and compare known upper bounds for the sizes of constant dimension codes. Here we assume that $v$, $d$, and $k$ are
integers with $2\le k\le v-2$, $4\le d\le 2\min\{k,v-k\}$, and $d$ even in all subsequent results. The bound $0\le k\le v$ just
ensures that {\Gqnk} is non-empty. Note that $\sdist(U,W)\le 2\min\{k,v-k\}$ and $\sdist(U,W)$ is even for all $U,W\in{\Gqnk}$. Restricting
to the set case, we trivially have $A_q(v,d;k)=\# {\Gqnk}=\gaussmnum{v}{k}{q}$ for $d\le 2$ or $k\le 1$, so that we assume $k\ge 2$ and $d\ge 4$,
which then implies $k\le v-2$ and $v\ge 4$. We remark that some of the latter bounds are also valid for parameters outside the ranges of non-trivial
parameters considered by us. Since the maximum size of a code with certain parameters is always an integer and some of the latter upper bounds can
produce non-integer values, we may always round them down. To ease the notation we will commonly omit the final rounding step.

The list of known bounds has not changed much since \cite{khaleghi2009subspace}, see also \cite{MR2810308}. Comparisons of those bounds
are scattered among the literature and partially hidden in comments, see e.g.\ \cite{MR3063504}. Additionally some results turn out to be
wrong or need a reinterpretation at the very least.

Counting $k$-dimensional subspaces having a \textit{large} intersection with a fixed $m$-dimensional
subspace gives:
\begin{lemma}
  \label{lemma_gsphere}
  For integers $0\le t\le k\le v$ and $k-t\le m\le v$ we have
  \[
    \# \left\{U\in{\Gqnk} \mid \dim(U\cap W)\ge k-t \right\}=
    \sum_{i=0}^{t} q^{(m+i-k)i} \gaussmnum{m}{k-i}{q} \gaussmnum{v-m}{i}{q},
  \]
  where $W\le V$ and $\dim(W)=m$.
\end{lemma}
\begin{proof}
  Let us denote $\dim(U\cap W)$ by $k-i$, where $\max\{0,k-m\}\le i\le \min\{t,v-m\}$. With this, the number of choices for $U$ is given by
  \begin{eqnarray*}
    &&\frac{\left(q^{m}-q^0\right)\cdot \left(q^{m}-q^1\right)\cdots\left(q^{m}-q^{k-i-1}\right)\cdot\left(q^{v}-q^{m+1}\right)\cdots \left(q^{v}-q^{m+i-1}\right)}
    {\left(q^k-q^{0}\right)\cdot \left(q^k-q^{1}\right)\cdots \left(q^k-q^{k-1}\right)}\\
    &=&
    \gaussmnum{m}{k-i}{q}\cdot \frac{\left(q^m\right)^i}{\left(q^{k-i}\right)^i}
    \cdot \gaussmnum{v-m}{i}{q}
    =q^{(m+i-k)i} \gaussmnum{m}{k-i}{q} \gaussmnum{v-m}{i}{q}.
  \end{eqnarray*}
  Finally apply the convention $\gaussmnum{a}{b}{q}=0$ for integers with $b<0$ or $b>a$.
  
\end{proof}

Note that $\dim(U\cap W)\ge k-t$ is equivalent to $\sdist(U,W)\le m-k+2t$.
The fact that the Grassmann graph is distance-regular implies:
\begin{theorem}(\textbf{Sphere-packing bound)}\cite[Theorem~6]{MR2451015}
  \label{thm_sphere_packing}
  \[
    A_q(v,d;k)\le \frac{\gaussmnum{v}{k}{q}}{\sum\limits_{i=0}^{\left\lfloor (d/2-1)/2\right\rfloor} q^{i^2} \gaussmnum{k}{i}{q} \gaussmnum{v-k}{i}{q}}
  \]
\end{theorem}
We remark, that we can obtain the denominator of the formula of Theorem~\ref{thm_sphere_packing} by setting $m=k$, $2t=d/2-1$ in
Lemma~\ref{lemma_gsphere} and applying $\gaussmnum{k}{k-i}{q}=\gaussmnum{k}{i}{q}$. The right hand side is symmetric with respect to orthogonal
subspaces, i.e., the mapping $k\mapsto v-k$ leaves it invariant.

By defining a puncturing operation one can decrease the dimension of the ambient space and the codewords. Since the minimum distance decreases
by at most two, we can iteratively puncture $d/2-1$ times, so that $A_q(v,d;k)\le \gaussmnum{v-d/2+1}{k-d/2+1}{q}= \gaussmnum{v-d/2+1}{v-k}{q}$
since $A_q(v',2;k')=\gaussmnum{v'}{k'}{q}$. Considering either the code or its orthogonal code gives:
\begin{theorem}(\textbf{Singleton bound)}\cite[Theorem~9]{MR2451015}
  \[
    A_q(v,d;k)\le \gaussmnum{v-d/2+1}{\max\{k,v-k\}}{q}
  \]
\end{theorem}
Referring to \cite{MR2451015} the authors of \cite{khaleghi2009subspace} state that even a relaxation of the Singleton bound is always stronger
than the sphere packing bound for non-trivial codes. However, for $q=2$, $v=8$, $d=6$, and $k=4$, the sphere-packing bound gives an upper bound
of $200787/451\approx 445.20399$ while the Singleton bound gives an upper bound of $\gaussmnum{6}{4}{2}=651$. For $q=2$, $v=8$, $d=4$, and $k=4$
it is just the other way round, i.e., the Singleton bound gives $\gaussmnum{7}{3}{2}=11811$ and the sphere-packing bound gives
$\gaussmnum{8}{4}{2}=200787$. Examples for the latter case are easy to find. For $d=2$ both bounds coincide and for $d=4$ the Singleton bound
is always stronger than the sphere-packing bound since $\gaussmnum{v-1}{k}{q}<\gaussmnum{v}{k}{q}$. The asymptotic bounds
\cite[Corollaries 7 and 10]{MR2451015}, using normalized parameters, and \cite[Figure 1]{MR2451015} suggest that there is only a small range of
parameters where the sphere-packing bound can be superior to the Singleton bound.\footnote{By a tedious computation one can check that the
sphere-packing bound is strictly tighter than the Singleton bound iff $q=2$, $v=2k$ and $d=6$.}

Given an arbitrary metric space $X$, an anticode of diameter $e$ is a subset whose elements have pairwise distance at most $e$. Since the
$q$-Johnson scheme is an association scheme the Anticode bound of Delsarte \cite{delsarte1973algebraic} can be applied. As a standalone argument
we go along the lines of \cite{ahlswede2009error} and consider bounds for codes on transitive graphs. By double-counting the number of pairs
$(a,g)\in A\cdot\operatorname{Aut}(\Gamma)$, where $g(a)\in B$, we obtain:
\begin{lemma}
  \cite[Lemma~1]{ahlswede2009error}, cf.~\cite[Theorem~1']{ahlswede2001perfect} Let $\Gamma=(V,E)$ be a graph that admits a transitive group of automorphisms $\operatorname{Aut}(\Gamma)$ and
  let $A,B$ be arbitrary subsets of the vertex set $V$. Then, there exists a group element $g\in\operatorname{Aut}(\Gamma)$ such that
  \[
    \frac{|g(A)\cap B|}{|B|}\ge \frac{|A|}{|V|}.
  \]
\end{lemma}

\begin{corollary}
  \label{cor_ahlswede}
  \cite[Corollary~1]{ahlswede2009error}, cf.~\cite[Theorem~1]{ahlswede2001perfect} Let $\mathcal{C}_D\subseteq {\Gqnk}$ be a code with (injection or graph)
  distances from $D=\{d_1,\dots,d_s\}\subseteq \{1,\dots,v\}$. Then, for an arbitrary subset $\mathcal{B}\subseteq{\Gqnk}$ there exists a code
  $\mathcal{C}_D^*(\mathcal{B})\subseteq\mathcal{B}$ with distances from $D$ such that
  \[
    \frac{\left|\mathcal{C}_D^*(\mathcal{B})\right|}{\left|\mathcal{B}\right|}\ge\frac{\left|\mathcal{C}_D\right|}{\gaussmnum{v}{k}{q}}.
  \]
\end{corollary}

If $\mathcal{C}_D\subseteq {\Gqnk}$ is a constant dimension code with minimum injection distance $d$, i.e., $D=\{d,\dots,v\}$, and $\mathcal{B}$ is an
anticode with diameter $d-1$, we have $\# \mathcal{C}_D^*(\mathcal{B})=1$, so that we obtain Delsarte's Anticode bound
\begin{equation}
  \#\mathcal{C}_D\le \frac{\gaussmnum{v}{k}{q}}{\# \mathcal{B}}.
\end{equation}

The set of all elements of {\Gqnk} which contain a fixed $(k-d/2+1)$-dimensional subspace is an anticode of diameter $d-2$ with
$\gaussmnum{v-k+d/2-1}{d/2-1}{q}$ elements. By orthogonality, the set of all elements of {\Gqnk} which are contained in a fixed
$(k+d/2-1)$-dimensional subspace is also an anticode of diameter $d-2$ with $\gaussmnum{k+d/2-1}{k}{q}=\gaussmnum{k+d/2-1}{d/2-1}{q}$ elements. Frankl and Wilson proved
in \cite[Theorem~1]{MR867648} that these anticodes have the largest possible size, which implies:
\begin{theorem}
  \label{thm_anticode}
  (\textbf{Anticode bound})
  \[
    A_q(v,d;k)\le \frac{\gaussmnum{v}{k}{q}}{\gaussmnum{\max\{k,v-k\}+d/2-1}{d/2-1}{q}}
  \]
\end{theorem}
Using different arguments, Theorem~\ref{thm_anticode} was proved in \cite[Theorem~5.2]{wang2003linear} by  Wang, Xing, Safavi-Naini
in 2003. Codes that can achieve the (unrounded) value \linebreak $\gaussmnum{v}{k}{q}/\gaussmnum{\max\{k,v-k\}+d/2-1}{d/2-1}{q}$ are called
Steiner structures. It is a well-known and seemingly very hard problem to decide whether a Steiner structure for $v=7$, $d=4$, and $k=3$ exists.
For $q=2$ the best known bounds are $333\le A_2(7,4;3)\le 381$. Additionally it is known that a code attaining the upper bound can have automorphisms of
at most order $2$, see \cite{fano_aut}. So far, the only known Steiner structure corresponds to $A_2(13,4;3)=1597245$ \cite{Braun16}.
The reduction to Delsarte's Anticode bound can be found e.g.\ in \cite[Theorem~1]{MR2810308}.

Since the sphere underlying the proof of Theorem~\ref{thm_sphere_packing} is also an anticode, Theorem~\ref{thm_sphere_packing} is implied by
Theorem~\ref{thm_anticode}. For $d=2$ both bounds coincide. In \cite[Section 4]{xia2009johnson} Xia and Fu verified that the Anticode bound is
always stronger than the Singleton bound for the ranges of parameters considered by us.

Mimicking a classical bound of Johnson on binary error-correcting codes with respect to the Hamming distance, see \cite[Theorem~3]{johnson1962new}
and also \cite{tonchev1998codes}, Xia and Fu proved:
\begin{theorem}
  \label{thm_johnson_I}
  (\textbf{Johnson type bound I}) \cite[Theorem~2]{xia2009johnson}\\ If $\left(q^k-1\right)^2>\left(q^v-1\right)\left(q^{k-d/2}-1\right)$, then
  \[
    A_q(v,d;k)\le \frac{\left(q^k-q^{k-d/2}\right)\left(q^v-1\right)}{\left(q^k-1\right)^2-\left(q^v-1\right)\left(q^{k-d/2}-1\right)}.
  \]
\end{theorem}
However, the required condition of Theorem~\ref{thm_johnson_I} is rather restrictive and can be simplified considerably.

\begin{proposition}
  \label{prop_johnson_I}
  For $0\le k<v$, the bound in Theorem~\ref{thm_johnson_I} is applicable iff $d=2\min\{k,v-k\}$ and $k\ge 1$. Then, it is equivalent to
  \[
    A_q(v,d;k)\le \frac{q^v-1}{q^{\min\{k,v-k\}}-1}.
  \]
\end{proposition}
\begin{proof}
  If $k=0$ we have $\left(q^k-1\right)^2=0$, so that we assume $k\ge 1$ in the following. If $k\le v-k$ and $d\le 2k-2$, then
  \[
    \left(q^v\!-\!1\right)\left(q^{k\!-\!d/2}\!-\!1\right)\ge \left(q^{2k}\!-\!1\right)\left(q\!-\!1\right)\ge q^{2k}\!-\!1
    \overset{q\ge 2,k\ge 1}{>} q^{2k}\!-\!2q^k\!+\!1=\left(q^k\!-\!1\right)^2.
  \]
  If $k\ge v-k+1$ and $d\le 2v-2k-2$, then
  \[
    \left(q^v\!-\!1\right)\left(q^{k\!-\!d/2}\!-\!1\right)\ge \left(q^v\!-\!1\right)\left(q^2\!-\!1\right)\overset{q\ge 2,v\ge 1}{>}
    \left(q^{(v+1)/2}-1\right)^2\ge \left(q^k\!-\!1\right)^2.
  \]
  If $d=2\min\{k,v-k\}$, $q\ge 2$, and $k\ge 1$, then it can be easily checked that the condition of Theorem~\ref{thm_johnson_I} is satisfied and
  we obtain the proposed formula after simplification.
  
\end{proof}
For $k=v$ Theorem~\ref{thm_johnson_I} gives $A_q(v,d;v)\le 1$ which is trivially satisfied with equality.
In Subsection~\ref{subsec_bounds_partial_spreads} we will provide tighter upper bounds for the special case where $d=2k$, i.e., partial spreads.
Indeed, the bound stated in Proposition~\ref{prop_johnson_I} corresponds
to the most trivial upper bounds for partial spreads that is tight iff $k$ divides $v$, as we will see later on. So, due to orthogonality,
Theorem~\ref{thm_johnson_I} is dominated by the partial spread bounds discussed later on.

While the previously mentioned generalization of a classical bound of Johnson on binary error-correcting codes yields the rather weak
Theorem~\ref{thm_johnson_I}, generalizing \cite[Inequality~(5)]{johnson1962new}, see \cite{xia2009johnson} yields a very strong upper bound:
\begin{theorem}
  \label{thm_johnson_II}
  (\textbf{Johnson type bound II}) \cite[Theorem~3]{xia2009johnson}, \cite[Theorem~4,5]{MR2810308}
  \begin{eqnarray}
  A_q(v,d;k) &\le&
  \frac{q^v-1}{q^k-1} A_q(v-1,d;k-1)
  \label{ie_j_2}\\
  A_q(v,d;k) &\le&
  \frac{q^v-1}{q^{v-k}-1} A_q(v-1,d;k)
  \label{ie_j_o}
  \end{eqnarray}
\end{theorem}
Note that for $d=2k$ Inequality~(\ref{ie_j_2}) gives $A_q(v,2k;k)\le \left\lfloor\frac{q^v-1}{q^k-1}\right\rfloor$ since we have $A_q(v-1,2k;k-1)=1$
by definition. Similarly, for $d=2(v-k)$, Inequality~(\ref{ie_j_o}) gives $A_q(v,2v-2k;k)\le \left\lfloor\frac{q^v-1}{q^{v-k}-1}\right\rfloor$.

Some sources like \cite[Theorem~3]{xia2009johnson} list just Inequality~\ref{ie_j_2} and omit
Inequality~\ref{ie_j_o}. This goes in line with the treatment of the classical Johnson type bound II for binary error-correcting codes,
see e.g.\ \cite[Theorem~4 on page 527]{MR0465510}, where the other bound is formulated as Problem~(2) on page 528 with the hint that ones should
be replaced by zeros. Analogously, we can consider orthogonal codes:
\begin{proposition}
  Inequality~(\ref{ie_j_2}) and Inequality~(\ref{ie_j_o}) are equivalent using orthogonality, cf.~\cite[Section III, esp.~Lemma~13]{MR2810308}.
\end{proposition}
\begin{proof}
  We have
  \begin{eqnarray*}
    A_q(v,d;k)&=&A_q(v,d;v-k)\overset{(\ref{ie_j_2})}{\le}
    \frac{q^v-1}{q^{v-k}-1} A_q(v-1,d;v-k-1)
    \\
    &=&
    \frac{q^v-1}{q^{v-k}-1} A_q(v-1,d;k),
  \end{eqnarray*}
  which is Inequality~(\ref{ie_j_o}), and
  \begin{eqnarray*}
    A_q(v,d;k)&=&A_q(v,d;v-k)\overset{(\ref{ie_j_o})}{\le}
    \frac{q^v-1}{q^{k}-1} A_q(v-1,d;v-k)
    \\
    &=&
    \frac{q^v-1}{q^{k}-1} A_q(v-1,d;k-1),
  \end{eqnarray*}
  which is Inequality~(\ref{ie_j_2}).
  
\end{proof}

Of course, the bounds in Theorem~\ref{thm_johnson_II} can be applied iteratively. In the classical Johnson space the optimal
order of the corresponding inequalities is unclear, see e.g.\ \cite[Research Problem 17.1]{MR0465510}. Denoting the maximum size
of a binary constant-weight block code of length $n$, Hamming distance $d$ and weight $k$ by $A(n,d,w)$, the two corresponding
variants of the inequalities in Theorem~\ref{thm_johnson_II} are $A(n,d,w) \le \lfloor n/w \cdot A(n-1,d,w-1) \rfloor$
and $A(n,d,w) \le \lfloor n/(n-w) \cdot A(n-1,d,w) \rfloor$. Applying the first bound yields
\[
  A(28,8,13) \le \lfloor 28/13 \cdot A(27,8,12) \rfloor \le \lfloor 28/13 \cdot 10547 \rfloor = 22716
\]
while applying the second bound yields
\[
  A(28,8,13) \le \lfloor 28/15 \cdot A(27,8,13) \rfloor \le \lfloor 28/15 \cdot 11981 \rfloor = 22364
\]
using the numerical bounds from \begin{center}\url{http://webfiles.portal.chalmers.se/s2/research/kit/bounds/cw.html}, cf.~\cite{agrell2000upper}.\end{center}
The authors of \cite{MR2810308,khaleghi2009subspace} state that the optimal choice of Inequality~(\ref{ie_j_2}) or
Inequality~(\ref{ie_j_o}) is unclear, too. However, this question is much easier to answer for constant dimension codes.
\begin{proposition}
  \label{prop_optimal_johnson}
  For $k \le v/2$ we have
  \[
    \left\lfloor \frac{q^v-1}{q^k-1} A_q(v-1,d;k-1) \right\rfloor
    \le
    \left\lfloor \frac{q^v-1}{q^{v-k}-1} A_q(v-1,d;k) \right\rfloor,
  \]
  where equality holds iff $v=2k$.
\end{proposition}
\begin{proof}
  By considering orthogonal codes we obtain equality for $v=2k$. Now we assume $k < v/2$ and show
  \begin{equation}
    \label{ie_wi_1}
    \frac{q^v-1}{q^k-1} A_q(v-1,d;k-1) + 1   \le   \frac{q^v-1}{q^{v-k}-1} A_q(v-1,d;k),
  \end{equation}
  which implies the proposed statement. Considering the size of the LMRD code we can lower bound the right hand side
  of Inequality~(\ref{ie_wi_1}) to
  \[
    \frac{q^v-1}{q^{v-k}-1} A_q(v-1,d;k) \ge \frac{q^v-1}{q^{v-k}}\cdot q^{(v-k-1)(k-d/2+1)}.
  \]
  Since
  \[
    \frac{\gaussmnum{v-1}{k-1}{q}}{\gaussmnum{v-k+d/2-1}{d/2-1}{q}}=\frac{\prod\limits_{i=1}^{k-1} \frac{q^{v-k+i}-1}{q^i-1}}
    {\prod\limits_{i=1}^{d/2-1} \frac{q^{v-k+i}-1}{q^i-1}}\le \prod\limits_{i=d/2}^{k-1} \frac{q^{v-k+i}}{q^i-1}=
    q^{(v-k)(k-d/2)} \prod\limits_{i=d/2}^{k-1} \frac{1}{1-q^{-i}}
  \]
  we can use the Anticode bound to upper bound the left hand side of Inequality~(\ref{ie_wi_1}) to
  \[
    \frac{q^v-1}{q^k-1} A_q(v-1,d;k-1) + 1\le \frac{q^v-1}{q^k-1}\cdot q^{(v-k)(k-d/2)}\cdot \mu(k-1,d/2,q)+1,
  \]
  where $\mu(a,b,q):=\prod\limits_{i=b}^{a} \left(1-q^{-i}\right)^{-1}$. Thus, it suffices to verify
  \begin{equation}
    \label{ie_wi_2}
    \frac{q^{k-d/2+1}}{q^k-1}\cdot \mu(k-1,d/2,q)+\frac{1}{f}
    \le 1,
  \end{equation}
  where we have divided by
  \[
    f:=\frac{q^v-1}{q^{v-k}}\cdot q^{(v-k-1)(k-d/2+1)}
    =\frac{q^v-1}{q}\cdot q^{(v-k-1)(k-d/2)}.
  \]
  Since $d\ge 4$, we have $\mu(k-1,d/2,q)\le  \prod\limits_{i=2}^{\infty} \left(1-q^{-i}\right)^{-1} \le
  \prod\limits_{i=2}^{\infty} \left(1-2^{-i}\right)^{-1}<1.74$. Since $v\ge 4$ and $q\ge 2$, we have $\frac{1}{f}\le \frac{2}{15}$.
  Since $k\ge 2$, we have $\frac{q^{k-d/2+1}}{q^k-1}\le \frac{q}{q^2-1}$, which is at most $\frac{3}{8}$ for $q\ge 3$. Thus,
  Inequality~(\ref{ie_wi_2}) is valid for all $q\ge 3$.

  If $d\ge 6$ and $q=2$, then $\mu(k-1,d/2,q)\le  \prod\limits_{i=3}^{\infty} \left(1-2^{-i}\right)^{-1} <1.31$ and
  $\frac{q^{k-d/2+1}}{q^k-1}\le \frac{1}{3}$, so that Inequality~(\ref{ie_wi_2}) is satisfied.

  In the remaining part of the proof we assume $d=4$ and $q=2$. If $k=2$, then $\mu(k-1,d/2,q)=1$ and $\frac{q^{k-d/2+1}}{q^k-1}=\frac{2}{3}$.
  If $k=3$, then $\mu(k-1,d/2,q)=\frac{4}{3}$ and $\frac{q^{k-d/2+1}}{q^k-1}=\frac{4}{7}$. If $k\ge 4$, then
  $\frac{q^{k-d/2+1}}{q^k-1}\le \frac{8}{15}$, $\mu(k-1,d/2,q)\le 1.74$, and $\frac{1}{f}\le \frac{2}{255}$ due to $v\ge 2k\ge 8$.
  Thus, Inequality~(\ref{ie_wi_2}) is valid in all cases.
  
\end{proof}

Knowing the optimal choice between Inequality~(\ref{ie_j_2}) and Inequality~(\ref{ie_j_o}), we can iteratively apply Theorem~\ref{thm_johnson_II}
in an ideal way initially assuming $k\le v/2$:
\begin{corollary}
\textbf{(Implication of the Johnson type bound II)}
\label{cor_johnson_opt}
\[
A_q(v,d;k)
\le
\left\lfloor \frac{q^{v}\!-\!1}{q^{k}\!-\!1} \left\lfloor \frac{q^{v\!-\!1}\!-\!1}{q^{k\!-\!1}\!-\!1} \left\lfloor \ldots
\left\lfloor \frac{q^{v\!-\!k\!+\!d/2\!+\!1}\!-\!1}{q^{d/2\!+\!1}\!-\!1} A_q(v\!-\!k\!+\!d/2,d;d/2) \right\rfloor
\ldots \right\rfloor \right\rfloor \right\rfloor
\]
\end{corollary}
We remark that this upper bound is commonly stated in an explicit version, where
$A_q(v\!-\!k\!+\!d/2,d;d/2)\le \left\lfloor\frac{q^{v-k+d/2}-1}{q^{d/2}-1}\right\rfloor$ is inserted, see e.g.\ \cite[Theorem~6]{MR2810308},
\cite[Theorem~7]{khaleghi2009subspace}, and \cite[Corollary~3]{xia2009johnson}. However, currently much better bounds for partial spreads
are available.

It is shown in \cite{xia2009johnson} that the Johnson bound of Theorem~\ref{thm_johnson_II} improves on the Anticode bound in
Theorem~\ref{thm_anticode}, see also~\cite{MR3063504}. To be more precise, removing the floors in the upper bound of
Corollary~\ref{cor_johnson_opt} and replacing $A_q(v-k+d/2,d;d/2)$ by $\frac{q^{v-k+d/2}-1}{q^{d/2}-1}$ gives
\[
  \prod_{i=0}^{k-d/2} \frac{q^{v-i}-1}{q^{k-i}-1}
  =
\frac{\prod_{i=0}^{k-1} \frac{q^{v-i}-1}{q^{k-i}-1}}{\prod_{i=k-d/2+1}^{k-1} \frac{q^{v-i}-1}{q^{k-i}-1}}
  =\frac{\gaussmnum{v}{k}{q}}{\gaussmnum{v-k+d/2-1}{d/2-1}{q}},
\]
which is the right hand side of the Anticode bound for $k\le v-k$. So, all upper bounds mentioned so far are (weakly) dominated by
Corollary~\ref{cor_johnson_opt}, if we additionally assume $k\le v-k$.
As a possible improvement \cite[Theorem~3]{ahlswede2009error} was mentioned as \cite[Theorem~8]{khaleghi2009subspace}. Here, we
correct typos and give a slightly enlarged proof, thanks to a personal communication with Aydinian.
\begin{theorem}
  \label{thm_ahlswede}
  \cite[Theorem~3]{ahlswede2009error}
  For integers $0\le t< r\le k$, $k-t\le m\le v$, and $t\le v-m$ we have
  \[
    A_q(v,2r;k)\le \frac{\gaussmnum{v}{k}{q} A_q(m,2r-2t;k-t)}{\sum_{i=0}^t q^{i(m+i-k)}\gaussmnum{m}{k-i}{q}\gaussmnum{v-m}{i}{q}}.
  \]
\end{theorem}
\begin{proof}
  Let $W$ be a fixed subspace with $\dim(W)=m$ and define
  \[
    \mathcal{B}=\left\{U\in{\Gqnk} \mid \dim(U\cap W)\ge k-t \right\},
  \]
  so that $\#\mathcal{B}$ is given by Lemma~\ref{lemma_gsphere}.
  Consider a $(v,\#\mathcal{C}^*,d;k)$ code $\mathcal{C}^*\subseteq \mathcal{B}$ and take $\mathcal{C}':=\mathcal{C}^*\cap W$ noting
  that the latter has a minimum distance of at least $2r-2t$. Two arbitrary codewords $U_1\neq U_2\in\mathcal{C}'$ have distance
  $\sdist(U_1,U_2)\ge 2r-2t+i+j$, where we write $\dim(U_1)=k-t+i$ and $\dim(U_2)=k-t+j$ for integers $0\le i,j\le t$.
  Replacing each codeword of $\mathcal{C}'$ by an arbitrary $k-t$-dimensional subspace, we obtain a cdc $\mathcal{C}$ with a minimum
  distance of at least $2r-2t$. Since $t<r$ we have $\#\mathcal{C}^*=\#\mathcal{C}'=\#\mathcal{C}$, so that Corollary~\ref{cor_ahlswede}
  gives the proposed upper bound.
  
\end{proof}
As Theorem~\ref{thm_ahlswede} has quite some degrees of freedom, we partially discuss the optimal choice of parameters.
For $t=0$ and $m\le v-1$, we obtain $A_q(v,d;k)\le \gaussmnum{v}{k}{q}/\gaussmnum{m}{k}{q}\cdot A_q(m,d;k)$, which is
the $(v-m)$-fold iteration of Inequality~(\ref{ie_j_o}) of the Johnson bound (without rounding). Thus, $m=v-1$ is the best choice
for $t=0$, yielding a bound that is equivalent to Inequality~(\ref{ie_j_o}). For $t=1$ and $m=v-1$ the bound can be rewritten to
$A_q(v,d;k)\le  A_q(v-1,d-2;k-1)$, see the proof of Proposition~\ref{prop_ahlswede_partial_spread}. For $t> v-m$ the bound remains
valid but is strictly weaker than
for $t=v-m$. Choosing $m=v$ gives the trivial bound $A_q(v,2r;k)\le A_q(m,2r-2t;k-t)$.
For the range of parameters $2\le q\le 9$, $4\le v\le 100$, limited facing nerve-jangling numerical pitfalls, and
$4\le d\le2k\le v$, where $q$ is of course a prime power and $d$ is even, the situation is as follows. If $d\neq 2k$, there are no proper
improvements with respect to Theorem~\ref{thm_johnson_II}. For the case $d=2k$, i.e., partial
spreads treated in the next subsection, we have some improvements compared to $\lfloor(q^v-1)/(q^k-1)\rfloor$ which is the most trivial bound for partial spreads.
Within our numerical range, most of them are
covered by the following proposition, where we apply Theorem~\ref{thm_ahlswede} with $t=1$ and $m=v-1$ to $A_q(v,2k;k)$.
The other cases are due to the fact that Theorem~\ref{theorem_ps_bound_2} is tighter than Theorem~\ref{theorem_ps_bound_1} for larger values of $z$.
In no case a proper improvement with respect to the tighter bounds from the next subsection emerged.

\begin{proposition}
  \label{prop_ahlswede_partial_spread}
    For $w \ge 1$ and $k \ge q^w+3$ we have $A_q(2k+w,2k;k)\le$
  \[
\left\lfloor\frac{\gaussmnum{2k+w}{k}{q} A_q(2k+w-1,2k-2;k-1)}{\sum_{i=0}^{1}
q^{i(k+w-1+i)} \gaussmnum{2k+w-1}{k-i}{q} \gaussmnum{(2k+w)-(2k+w-1)}{i}{q}} \right\rfloor
<\left\lfloor\frac{q^{2k+w}-1}{q^{k}-1} \right\rfloor=q^{k+w}+q^w
\]
\end{proposition}
\begin{proof}
  Note that $k \ge q^w+3$ implies $w < k$. The left hand side simplifies to
  \[
\frac{\gaussmnum{2k+w}{k}{q} A_q(2k+w-1,2k-2;k-1)}{\sum_{i=0}^{1}
q^{i(k+w-1+i)} \gaussmnum{2k+w-1}{k-i}{q} \gaussmnum{(2k+w)-(2k+w-1)}{i}{q}}
=
A_q(2k+w-1,2k-2;k-1)
.
  \]
  Then we apply Theorem~\ref{theorem_ps_bound_1} with $t=2$, $r=w+1$, and $z=\gaussmnum{w}{1}{q}-1$,
  which yields $A_q(2k+w-1,2k-2;k-1)\le q^{k+w}+1+q^w-q<q^{k+w}+q^w$ for $k-1\ge q^w+2$.

\end{proof}
We remark that applying Theorem~\ref{theorem_ps_bound_2} and Theorem~\ref{theorem_ps_bound_1} directly is at least as good as the application of Theorem~\ref{thm_ahlswede}
with $t=1$ and $m=v-1$ for $d=2k$.

The Delsarte linear programming bound for the $q$-Johnson scheme was obtained in \cite{delsarte1978hahn}. However, numerical computations indicate
that it is not better than the Anticode bound, see \cite{MR3063504}. For $d\neq 2\min\{k,v-k\}$, i.e., the non-partial spread case, besides the
stated bound only the following two specific bounds, based on extensive computer calculations, are known:
\begin{theorem}
  \label{thm_specific_bound_1}
  \cite[Theorem~1]{MR3329980} $A_2(6,4;3)=77$
\end{theorem}
\begin{proposition}
  \label{prop_specific_bound_2}
  \cite{heinlein2017new} $A_2(8,6;4)\le 272$
\end{proposition}
As the authors of \cite{heinlein2017new} have observed, the Johnson bound of Theorem~\ref{thm_johnson_II} does not improve
upon Corollary~\ref{cor_johnson_opt} when applied to Theorem~\ref{thm_specific_bound_1} or Proposition~\ref{prop_specific_bound_2}.

If we additionally restrict ourselves to constant dimension codes, that contain a lifted MRD code, another upper bound is known:
\begin{theorem}\cite[Theorem~10~and~11]{MR3015712}\label{theo:MRD_upper_bound}
Let $\mathcal{C}\subseteq \gaussmset{\mathbb{F}_q^v}{k}$ be a constant dimension code, with $v\ge 2k$ and minimum
subspace distance $d$, that contains a lifted MRD code.
\begin{itemize}
\item If $d=2(k-1)$ and $k \ge 3$, then $\# \mathcal{C} \le q^{2(v-k)} + A_q(v-k,2(k-2);k-1)$;
\item if $d=k$, where $k$ is even, then $\#\mathcal{C} \le q^{(v-k)(k/2+1)} +
      \gaussmnum{v-k}{k/2}{q}\frac{q^v-q^{v-k}}{q^{k}-q^{k/2}} + A_q(v-k,k;k)$.
\end{itemize}
\end{theorem}

\subsection{Upper bounds for partial spreads}
\label{subsec_bounds_partial_spreads}
The case of constant dimension codes with maximum possible subspace distance $d=2k$ is known under the name partial spreads. Counting points,
i.e., $1$-dimensional subspaces, in $\F_q^v$ and $\F_q^k$ gives the obvious upper bound
$A_q(v,2k;k)\le \gaussmnum{v}{1}{q}/\gaussmnum{k}{1}{q}=\left(q^v-1\right)/\left(q^k-1\right)$. In the case of
equality one speaks of spreads, for which a handy existence criterion is known from the work of Segre in 1964.
\begin{theorem}{\cite[\S VI]{segre1964teoria}}
  \label{thm_spread} $\mathbb{F}_q^v$ contains a spread if and only if $k$ is a
  divisor of $v$.
\end{theorem}

If $k$ is not a divisor of $v$, far better bounds are known including some recent improvements, which we will briefly summarize.
For a more detailed treatment we refer to e.g.\ \cite{honold2016partial}. The best known parametric construction was given by
Beutelspacher in 1975:
\begin{theorem}\cite{beutelspacher1975partial}
  \label{thm:multicomponent}
  For positive integers $v,k$ satisfying $v=tk+r$, $t\geq 2$ and
  $1\leq r\leq k-1$ we have
  $A_q(v,2k;k)\geq 1+\sum_{i=1}^{t-1}q^{ik+r}=\frac{q^v-q^{k+r}+q^k-1}{q^k-1}$ with equality for $r=1$.
\end{theorem}

The determination of $A_2(v,6;3)$ for $v\equiv 2\pmod 3$ was achieved more than 30 years later in \cite{spreadsk3} and
continued to $A_2(v,2k;k)$ for $v\equiv 2\pmod k$ and arbitrary $k$ in \cite{kurzspreads}. Besides the parameters of
$A_2(8+3l,6;3)$, for $l\ge 0$, see \cite{spreadsk3} for an example showing $A_2(8,6;3)\ge 34$, no partial spreads exceeding the
lower bound from Theorem~\ref{thm:multicomponent} are known.

For a long time the best known upper bound on $A_q(v,2k;k)$ was the one obtained by Drake and
Freeman in 1979:
\begin{theorem}\cite[Corollary~8]{nets_and_spreads}
  \label{thm_partial_spread_4}
  If $v=kt+r$ with $0<r<k$, then
  \[
    A_q(v,2k;k)\le \sum_{i=0}^{t-1} q^{ik+r} -\left\lfloor\theta\right\rfloor-1
    =q^r\cdot \frac{q^{kt}-1}{q^k-1}-\left\lfloor\theta\right\rfloor-1,
  \]
  where $2\theta=\sqrt{1+4q^k(q^k-q^r)}-(2q^k-2q^r+1)$.
\end{theorem}

Quite recently this bound has been generalized to:
\newcommand{\uu}{\lambda}
\begin{theorem} \cite[Theorem~2.10]{kurz2017packing}
  \label{theorem_ps_bound_2}
  For integers $r\ge 1$, $t\ge 2$, $y\ge \max\{r,2\}$,
  $z\ge 0$ with $\uu=q^{y}$, $y\le k$,
  $k=\gaussmnum{r}{1}{q}+1-z>r$, $v=kt+r$, and  $l=\frac{q^{v-k}-q^r}{q^k-1}$, we have $A_q(v,2k;k)\le $
  $
     lq^k+\left\lceil \uu -\frac{1}{2}-\frac{1}{2}
    \sqrt{1+4\uu\left(\uu-(z+y-1)(q-1)-1\right)} \right\rceil
  $.
\end{theorem}

The construction of Theorem~\ref{thm:multicomponent} is asymptotically optimal for
$k\gg r=v\bmod k$, as recently shown by N{\u{a}}stase and Sissokho:
\begin{theorem}\cite[Theorem~5]{nastase2016maximum}
  \label{thm_partial_spread_asymptotic}
  Suppose $v=tk+r$ with $t\geq 1$ and $0<r<k$.
  If $k>\gaussmnum{r}{1}{q}$
  then   $A_q(v,2k;k)=1+\sum_{i=1}^{t-1}q^{ik+r}=\frac{q^{v}-q^{k+r}+q^k-1}{q^k-1}$.
\end{theorem}

Applying similar techniques, the result was generalized to $k\le \gaussmnum{r}{1}{q}$:
\begin{theorem}
  \label{theorem_ps_bound_1} \cite[Theorem~2.9]{kurz2017packing}
  For integers $r\ge 1$, $t\ge 2$, $u\ge 0$, and $0\le z\le \gaussmnum{r}{1}{q}/2$ with $k=\gaussmnum{r}{1}{q}+1-z+u>r$ we have
  $A_q(v,2k;k)\le lq^k+1+z(q-1)$, where $l=\frac{q^{v-k}-q^r}{q^k-1}$ and $v=kt+r$.
\end{theorem}
Using Theorem~\ref{theorem_ps_bound_2} the restriction $z\le \gaussmnum{r}{1}{q}/2$ can be removed from Theorem~\ref{theorem_ps_bound_1},
see \cite{honold2016partial}.

Currently, Theorem~\ref{thm_spread}, Theorem~\ref{theorem_ps_bound_2}, and Theorem~\ref{theorem_ps_bound_1} constitute the tightest parametric
bounds for $A_q(v,2k;k)$. The only known improvements, by exactly one in every case, are given by the $21$ specific bounds stated in
\cite{kurz2017packing}, which are based on the linear programming method applied to projective $q^{k-1}$-divisible linear error-correcting
codes over $\mathbb{F}_q$ with respect to the Hamming distance, see~\cite{honold2016partial}. As this connection seemed to be overlooked
before, it may not be improbable that more sophisticated methods from classical coding theory can improve further values, which then imply improved
upper bounds for constant dimension codes via the Johnson bound of Theorem~\ref{thm_johnson_II}.

\section{The linkage construction revisited}
\label{sec_linkage}

A very effective and widely applicable construction of constant dimension codes was stated by Gluesing-Luerssen and
Troha:
\begin{theorem}\cite[Theorem~2.3]{MR3543532}, cf.~\cite[Corollary~39]{silberstein2015error}
\label{thm_original_linkage}
Let $C_i$ be a $(v_i, N_i, d_i; k)_q$ constant dimension code for $i=1,2$ and let $C_r$ be a $(k \times v_2, N_r, d_r)_q$ linear
rank metric code. Then \[\{\tau^{-1}(\tau(U)\mid M) : U \in C_1, M \in C_r\} \cup \{\tau^{-1}(0_{k \times v_1}| \tau(W)) : W \in C_2\}\]
is a $(v_1+v_2, N_1 N_R + N_2, \min\{d_1,d_2,2d_r\}; k)_q$ constant dimension code.
\end{theorem}
Here $A|B$ denotes the concatenation of two matrices with the same number of rows and $0_{m \times n}$ denotes the $m \times n$-matrix
consisting entirely of zeros. The resulting code depends on the choice of the codes $C_1$, $C_2$, $C_r$ and their representatives within isomorphism classes, so
that one typically obtains many isomorphism classes of codes with the same parameters.

We remark that \cite[Theorem~37]{silberstein2015error} corresponds to the weakened version of Theorem~\ref{thm_original_linkage}
where the codewords from the cdc $C_2$ are not taken into account, cf. \cite[Theorem~5.1]{gluesing2015cyclic}. In
\cite[Corollary~39]{silberstein2015error} Silberstein and (Horlemann-)Trautmann obtain the same lower bound, assuming $d_1=d_2=2d_r$,
which is indeed the optimal choice, and $3k \le v$.\footnote{\label{footnote_linkage}It can be verified that for $2k \le v \le 3k-1$ the optimal choice of
$\Delta$ in \cite[Corollary39]{silberstein2015error} is given by $\Delta=v-k$. In that case the construction is essentially the union
of an LMRD code with an $(v-k,\#\mathcal{C}',d;k)_q$ code $\mathcal{C}'$. Note that for $v-k < \Delta \le v$ the constructed code is
an embedded $(\Delta,\#\mathcal{C}',d;k)_q$ code $\mathcal{C}'$.}

The main idea behind Theorem~\ref{thm_original_linkage} is to consider two sets of codewords with
disjoint pivot vectors across the two sets and to utilize the interplay between the rank and the subspace distance for a product type
construction. Using Lemma~\ref{lemma_d_s_d_h} the restriction of the disjointness of the pivot vectors can be
weakened, which gives the following improvement:

\begin{theorem}\label{theo:improved_linkage}
Let $C_i$ be a $(v_i, N_i, d_i; k)_q$ constant dimension code for $i=1,2$, $d \in 2 \mathbb{N}_{\ge 0}$ and let $C_r$ be a
$(k \times (v_2-k+d/2), N_r, d_r)_q$ linear rank metric code. Then
\[\mathcal{C}=\{\tau^{-1}(\tau(U)\mid M) : U \in C_1, M \in C_r\} \cup \{\tau^{-1}(0_{k \times (v_1-k+d/2)}| \tau(W)) : W \in C_2\}\]
is a $(v_1+v_2-k+d/2, N_1 N_R + N_2, \min\{d_1,d_2,2d_r,d\}; k)_q$ constant dimension code.
\end{theorem}
\begin{proof}
  The dimension of the ambient space and the codewords of $\mathcal{C}$ directly follow from the construction. Since the constructed
  matrices all are in rref and pairwise distinct, $\mathcal{C}$ is well defined and we have $\#\mathcal{C}=N_1 N_R + N_2$. It remains to
  lower bound the minimum subspace distance of $\mathcal{C}$.

  Let $A,C\in C_1$ and $B,D\in C_r$. If $A\neq C$, we have
  \begin{align*}
  \sdist(\tau^{-1}((\tau(A)\mid B)), \tau^{-1}((\tau(C)\mid D))) = 2\left(\rk
  \begin{pmatrix}
  \tau(A) & B\\
  \tau(C) & D
  \end{pmatrix}
  -k\right)
  \\
  \ge
  2\left(\rk
  \begin{pmatrix}
  \tau(A)\\
  \tau(C)
  \end{pmatrix}
  -k\right) = \sdist(A, C) \ge d_1
  \end{align*}
  using Equation~(\ref{eq_d_s_rk}) in the first step.
  If $A=C$ but $B\neq D$, we have
  \begin{align*}
    \sdist(\tau^{-1}((\tau(A)\mid B)), \tau^{-1}((\tau(C)\mid D))) = 2\left(\rk
  \begin{pmatrix}
  \tau(A) & B\\
  \tau(C) & D
  \end{pmatrix}
  -k\right)
  \\
  \ge
  2\left(\rk
  \begin{pmatrix}
  \tau(A) & B\\
  0 & D-B
  \end{pmatrix}
  -k\right)
  =
  2(k+\rk(D-B)-k)
  \ge 2d_r
  \text{.}
  \end{align*}

  For $A'\ne C'\in C_2$ applying Equation~(\ref{eq_d_s_rk}) gives
  \[
    \sdist(\tau^{-1}(0_{k \times (v_1-k+d/2)} \mid \tau(A')), \tau^{-1}(0_{k \times (v_1-k+d/2)}\mid \tau(C')))= \sdist(A',C')\ge d_2.
  \]

  Last, for two codewords $U \in \{\tau^{-1}(\tau(U)\mid M) \mid U \in C_1, M \in C_r\}$ and
  $W \in \{\tau^{-1}(0_{k \times (v_1-k+d/2)}\mid \tau(W)) \mid W \in C_2\}$, we can use the shape of the pivot vectors and apply
  Lemma~\ref{lemma_d_s_d_h}. The pivot vector $p(U)$ has its $k$ ones in the first $v_1$ positions and the pivot vector $p(W)$ has
  its $k$ ones not in the first $v_1-k+d/2$ positions, so that the ones can coincide at most at the positions
  $\{v_1-k+d/2+1,\ldots,v_1\}$. Thus, $\dham(p(U),p(W)) \ge k-(k-d/2) + k-(k-d/2) = d$. Lemma~\ref{lemma_d_s_d_h} then gives
  $\sdist(U,W)\ge d$.
  
\end{proof}

An example where Theorem~\ref{theo:improved_linkage} yields a larger code than Theorem~\ref{thm_original_linkage} is e.g.\ given for the parameters
$q=2$, $v=7$, $d=4$, and $k=3$. In order to apply Theorem~\ref{thm_original_linkage} we have to choose $v_1+v_2=7$, $3 \le v_1 \le 4$, and
$3 \le v_2 \le 4$. For $v_1=3$ we obtain $\#C_1 \le A_2(3,4;3)=1$ and $\#C_2 \le A_2(4,4;3) = 1$. Since the size of the rank metric code is
bounded by $\left\lceil 2^{4(3-2+1)}\right\rceil=2^8$, the constructed code has a size of at most $1 \cdot 2^8+1 = 257$. For $v_1=4$ the roles of
$C_1$ and $C_2$ interchange. Since the size of the rank metric code is bounded by $\left\lceil 2^{3(3-2+1)} \right\rceil=2^6$, the constructed code has
a size of at most $1 \cdot 2^6 +1 = 65$. In Theorem~\ref{theo:improved_linkage} we can choose $d=4$, so that we can drop one column
of the zero matrix preceding the matrices of the second set of codewords, i.e., $v_1+v_2=7+1=8$. Choosing $v_1=3$ and $v_2=5$ we
can achieve $\#C_1=A_2(3,4;3)=1$ and $\#C_2= A_2(5,4;3) = 9$. Since the size of the rank metric code can attain
$\left\lceil 2^{4(3-2+1)} \right\rceil=2^8$ we can construct a code of size $1 \cdot 2^8 + 9 = 265$. While for these parameters sill larger
codes are known, the situation significantly changes in general. Considering the range of parameters $2\le q\le 9$, $4\le v\le 19$, and
$4\le d\le 2k\le v$, where $q$ is of course a prime power and $d$ is even, Theorem~\ref{thm_original_linkage} provides the best known
lower bound for $A_q(v,d;k)$ in $41.8$\% of the cases, while Theorem~\ref{theo:improved_linkage} provides the best known
lower bound in $65.6$\% of the cases. Since the sizes of both constructions can coincide, the sum of both fractions gives more than $100$\%.
In just $34.4$\% of the cases strictly superior constructions are known compared to Theorem~\ref{theo:improved_linkage}, where most of
them arose from the so-called Echelon-Ferrers construction or one of their variants, see \cite{HKKW2016Tables} and the corresponding
webpage.\footnote{Entries of type \texttt{improved\_linkage(m)} correspond to Corollary~\ref{cor_improved_linkage} with $m$ chosen
as parameter.}

If one is interested in codes of large size, then one should choose the parameters $d_1$, $d_2$, $d_r$, and $d$,
in Theorem~\ref{theo:improved_linkage}, as small as
possible in order to maximize the sizes $N_1$, $N_2$, and $N_r$, i.e., we can assume $d_1=d_2=2d_r=d$. Moreover, the codes $C_1$, $C_2$,
and $C_r$ should have the maximum possible size with respect to their specified parameters. For $C_r$ the maximum possible size is
$M(q,k,v_2+d/2,d)$ and for $C_i$ the maximum possible size is $A_q(v_1,d;k)$, where $i=1,2$.
\begin{corollary}
For positive integers $k\le \min\{v_1,v_2\}$ and $d\equiv 0\pmod 2$ we have
$
A_q(v_1+v_2-k+d/2,d;k)
\ge A_q(v_1,d;k) \cdot M(q,k,v_2+d/2,d) + A_q(v_2,d;k)
$.
\end{corollary}
Instead of $A_q(v_1,d;k)$ or $A_q(v_2,d;k)$ we may also insert any lower bound of these commonly unknown values. By a variable transformation
we obtain:
\begin{corollary}\label{cor_improved_linkage}
For positive integers $k \le m \le v-d/2$ and $d\equiv 0\pmod 2$ we have
$A_q(v,d;k) \ge A_q(m,d;k) \cdot M(q,k,v-m+k,d) + A_q(v-m+k-d/2,d;k)$.
\end{corollary}

For the parameters of spreads the optimal choice of the parameter $m$ in Corollary~\ref{cor_improved_linkage} can be determined analytically:
\begin{lemma}
  If $d=2k$ and $k$ divides $v$, then Corollary~\ref{cor_improved_linkage} gives $A_q(v,d;k)\ge \frac{q^v-1}{q^k-1}$ for all $m=k,2k,\ldots,v-k$
  and smaller sizes otherwise.
\end{lemma}
\begin{proof}
  Using $A_q(v',2k;k)=(q^{v'}-1)/(q^k-1)$ for all integers $v'$ being divisible by $k$, we obtain
  \begin{eqnarray*}
    A_q(v,d;k) &\ge& A_q(m,d;k) \cdot M(q,k,v-m+k,2k) + A_q(v-m,2k;k)\\
    &=& \frac{q^m-1}{q^k-1}\cdot q^{v-m}+ \frac{q^{v-m}-1}{q^k-1}=\frac{q^{v}-1}{q^k-1}
  \end{eqnarray*}
  if $k$ divides $m$. Otherwise, $A_q(m,2k;k) \le (q^m-1)/(q^k-1) -1$ gives a lower bound.
  
\end{proof}

We remark that the tightest implications of Corollary~\ref{cor_improved_linkage} can be evaluated by dynamic programming. To this end
we consider fixed parameters $q$, $d$, $k$ and use the abbreviations $a(n):=A_q(n,d;k)$ and $b(n):=M(q,k,n+k,d)$ for integers $n$, so that
the inequality of Corollary~\ref{cor_improved_linkage} reads
\begin{equation}
  \label{eq_recursion_linkage}
  a(v) \ge a(m) \cdot b(v-m) + a(v-m+k-d/2).
\end{equation}
For a given maximal value $v$ we initialize the values $a(n)$ for $1\le n\le v$ by the best known lower bounds for
$A_q(n,d;k)$ from other constructions. Then we loop over $n$ from $k$ to $v$ and eventually replace $a(n)$ by
\[\max\{ a(m) \cdot b(n-m) + a(n-m+k-d/2) \mid k \le m \le n-d/2 \}.\]

By an arithmetic progression we can use (\ref{eq_recursion_linkage}) in order to obtain a lower bound for $a(v)=A_q(v,d;k)$ given
just two initial $a(i)$-values.

\begin{proposition}\label{prop:nonrecursive_version}
  For positive integers $k\le v_0$, $2s\ge d$, and $l\ge 0$, we have
  \[a(v_0+ls) \ge  a(v_0)\cdot b(s)^l+a(s-d/2+k) \gaussmnum{l}{1}{b(s)}.\]
  If additionally,  $v_0 \ge 2k-d/2$ and $k \ge d/2$, then we have
  \[a(v_0+ls) \ge a(s+k-d/2) \cdot {(q^{k-d/2+1})}^{n_0-k+d/2} \gaussmnum{l}{1}{q^{s(k-d/2+1)}} + a(v_0).\]
\end{proposition}
\begin{proof}
  Using Inequality~(\ref{eq_recursion_linkage}) with $v=v_0+ls$ and $m=v_0+(l-1)s$ gives
  \[a(v_0+ls) \ge a(v_0+(l-1)s) \cdot b(s) + a(s+k-d/2).\]
  By induction, we obtain
  \[a(v_0+ls) \ge  a(v_0+(l-i)s)\cdot b(s)^i + a(s+k-d/2) \gaussmnum{i}{1}{b(s)}\]
  for all $0\le i\le l$.

  For the second part, applying Inequality~(\ref{eq_recursion_linkage}) with $v=v_0+ls$ and $m=s+k-d/2$ gives
  \[a(v_0+ls) \ge a(s+k-d/2) \cdot b(v_0+(l-1)s-k+d/2) + a(v_0+(l-1)s).\]
  By induction, we obtain
  \[a(v_0+ls) \ge a(s+k-d/2) \cdot \sum_{j=1}^{i} b(v_0+(l-j)s-k+d/2) + a(v_0+(l-i)s)\]
  for all $0\le i\le l$.

  If $v_0 \ge 2k-d/2$ and $k \ge d/2$, then \[b(v_0+(l-j)s-k+d/2)={(q^{k-d/2+1})}^{v_0+(l-j)s-k+d/2},\] so that
  \begin{align*}
\sum_{j=1}^{l} b(v_0+(l-j)s-k+d/2)
=
\sum_{j=1}^{l} {(q^{k-d/2+1})}^{v_0+(l-j)s-k+d/2}=
\\
{(q^{k-d/2+1})}^{v_0-k+d/2} \sum_{r=0}^{l-1} {(q^{s(k-d/2+1)})}^{r}
=
{(q^{k-d/2+1})}^{v_0-k+d/2} \gaussmnum{l}{1}{q^{s(k-d/2+1)}}.
\end{align*}

\end{proof}

\begin{example}
\label{ex_ap_1}
Using $A_2(13,4;3) = 1597245$ \cite{Braun16} and $A_2(7,4;3) \ge 333$ \cite{HKKW2016Tables}, applying Proposition~\ref{prop:nonrecursive_version}
with $s=6$ gives
\begin{align*}
A_2(13+6l,4;3)
\ge
4096^l \cdot 1597245 + 333 \cdot\frac{4096^l-1}{4095}
\end{align*}
and
\begin{align*}
A_2(13+6l,4;3)
\ge
333 \cdot 16777216 \cdot\frac{4096^l-1}{4095} + 1597245
\end{align*}
for all $l\ge 0$.
\end{example}
In the next section we will see that the first lower bound almost meets the Anticode bound.

We remark that Theorem~\ref{theo:improved_linkage} can be easily generalized to a construction based on a union of $m\ge 2$ sets of codewords.
\begin{corollary}\label{cor:multiple_linkage}
  For positive integers $k$, $m$, and $i=1,\ldots,m$ let
\begin{itemize}
\item $C_i$ be an $(v_i,N_i,d_i;k)_q$ constant dimension code,
\item $\delta_i \in \mathbb{N}_{\ge 0}$, $\delta_m=0$,
\item $C_i^R$ be a $(k \times v_i^R, N_i^R, d_i^R)_q$ linear rank metric code, where $v_i^R=\sum_{j=1}^{i-1}(v_j-\delta_j)$ and $i \ne 1$,
\item $C_1^R=\emptyset$, $v_1^R=0$, $N_1^R=1$, and $d_1^R=\infty$.
\end{itemize}
Then
\[ \bigcup_{i=1}^{m} \{\tau^{-1}(0_{k \times (v-v_i-v_i^R)} \mid \tau(U_i) \mid M_i) : U_i \in C_i, M_i \in C_i^R\} \]
is a $(v,N,d;k)_q$ constant dimension code with
\begin{itemize}
\item $v=\sum_{i=1}^{m} (v_i-\delta_i)$,
\item $N=\sum_{i=1}^{m}N_i \cdot N_i^R$, and
\item $d=\min\{d_i, 2d_i^R, 2(k-\delta_i) \mid i=1,\ldots,m\}$.
\end{itemize}
\end{corollary}
\begin{proof}
We prove by inductively applying Theorem~\ref{theo:improved_linkage} $m-1$ times.
Denote
\[\tilde{C}_i:=\{\tau^{-1}(0_{k \times (v-v_i-v_i^R)} \mid \tau(U_i) \mid M_i) : U_i \in C_i, M_i \in C_i^R\}\]
for $i=1,\ldots,m$, i.e., $\tilde{C}_i$ is a padded $(v_i+v_i^R,N_i \cdot N_i^R,\min\{d_i,2d_i^R\};k)_q$ constant dimension code.
Applying Theorem~\ref{theo:improved_linkage} for $\tilde{C}_1$ and $\tilde{C}_2$ with $d=2(k-\delta_1)$ yields a
$(v_1+v_2-\delta_1,N_1+N_2\cdot N_2^R,\min\{d_1,d_2,2d_2^R,2(k-\delta_1)\};k)_q$ constant dimension code.
If the first $m'$ codes, $\tilde{C}_1,\ldots,\tilde{C}_{m'}$ yield an
$(\sum_{i=1}^{m'} (v_i-\delta_i)+\delta_{m'},\sum_{i=1}^{m'}N_i \cdot N_i^R,\min\{d_i, 2d_i^R, 2(k-\delta_i) \mid i=1,\ldots,m'\};k)_q$
constant dimension code $\tilde{C}_{1,\ldots,m'}$, then performing Theorem~\ref{theo:improved_linkage} for this code and $\tilde{C}_{m'+1}$
with $d=2(k-\delta_{m'})$ yields an $(\sum_{i=1}^{m'} (v_i-\delta_i)+\delta_{m'} + n_2-\delta_{m'},\sum_{i=1}^{m'}N_i \cdot N_i^R+N_{m'+1}
\cdot N_{m'+1}^R,\min\{d_i, 2d_i^R, 2(k-\delta_i) \mid i=1,\ldots,m'+1\};k)_q$ constant dimension code.

\end{proof}

Since the proof uses multiple applications of Theorem~\ref{theo:improved_linkage} this code can be found by the
dynamic programming approach based on Theorem~\ref{theo:improved_linkage}, i.e., Corollary~\ref{cor:multiple_linkage} is redundant.
However, it can be used to prove:

\begin{corollary}[cf. {\cite[Theorem~4.6]{MR3543532}}]
Let $C^R$ be an $(k \times v_1+v_2,d)_q$ linear MRD code, where $k \le v_i$, for $i=1,2$ and let $C_i$ be an $(v_{i-2},N_i,2d;k)_q$
constant dimension codes for $i=3,4$. Then
\begin{align*}
&
\{\tau^{-1}(I_{k \times k} \mid A) \mid A \in C^R\}
\\
\cup
&
\{\tau^{-1}(0_{k \times k} \mid \tau(A) \mid 0_{k \times v_2}) \mid A \in C_3\}
\\
\cup
&
\{\tau^{-1}(0_{k \times k} \mid 0_{k \times v_1} \mid \tau(A)) \mid A \in C_4\}
\end{align*}
is a $(v_1+v_2+k,q^{(v_1+v_2)(k-d+1)}+N_3+N_4,2\min\{d,k\};k)_q$ constant dimension code. Note that $k < d$ implies $N_3, N_4 \le 1$.
\end{corollary}
\begin{proof}
Applying Corollary~\ref{cor:multiple_linkage} with
\begin{itemize}
\item $m=3$
\item $\bar{C}_1 = C_4$, $\bar{C}_2 = C_3$,
\item $\bar{C}_3 = \{\tau^{-1}(I_{k \times k})\}$ (i.e., an $(k,1,\infty;k)_q$ constant dimension code)
\item $\delta_1=\delta_2=\delta_3=0$
\item $\bar{C}_1^R = \emptyset$
\item $\bar{C}_2^R=\{0_{k \times v_2}\}$ (i.e., an $(k \times v_2,1,\infty)_q$ rank metric code)
\item $\bar{C}_3^R$ an $(k \times (v_1+v_2,d))_q$ MRD code
\end{itemize}
yields an $(v_1+v_2+k,N_4 + N_3 + q^{(v_1+v_2)(k-d+1)},2\min\{d,k\};k)_q$ constant dimension code:
\begin{align*}
&
\{\tau^{-1}(I_{k \times k} \mid M_3) : M_3 \in \bar{C}_3^R\}
\\
\cup
&
\{\tau^{-1}(0_{k \times k} \mid \tau(U_2) \mid 0_{k \times v_2}) : U_2 \in C_3\}
\\
\cup
&
\{\tau^{-1}(0_{k \times (v_1+k)} \mid \tau(U_1)) : U_1 \in C_4\}
\end{align*}

\end{proof}

We remark that $\{(A \mid B) : A \in C_1^R, B \in C_2^R \}$ is a $(k \times (v_1+v_2),d)_q$ linear MRD code,
since each codeword has $\rk(A \mid B) \ge \rk(A) \ge d$. The other direction is not necessarily true, e.g.,
$\left(\begin{smallmatrix}I_{k-1}\\0\end{smallmatrix} \mid 0 \mid \ldots \mid 0 \mid w\right)$, where $w$ is a non-zero column, cannot be split in two matrices $\left(\begin{smallmatrix}I_{k-1}\\0\end{smallmatrix} \mid 0 \mid \ldots \mid 0\right)$ and $\left(0 \mid \ldots \mid 0 \mid w\right)$
both having rank distance at least $d$ for $d \ge 2$. Hence, this corollary constructs codes of the same size as
Theorem~4.6 in \cite{MR3543532} but these codes are not necessarily equal.

\section{Asymptotic bounds}
\label{sec_asymptotic}

K\"otter and Kschischang  have stated the bounds
\[
  1< q^{-k(v-k)}\cdot\gaussmnum{v}{k}{q}<4
\]
for $0<k<v$ in \cite[Lemma~4]{MR2451015} for the $q$-binomial coefficients. They used this result in order to prove that the lifted MRD codes, they
spoke about linearized polynomials, have at least a size of a quarter of the Singleton bound if $v$ tends to infinity. Actually, they have derived
a more refined bound, which can best expressed using the so called $q$-Pochhammer symbol
$
  (a;q)_n:=\prod_{i=0}^{n-1} \left(1-aq^i\right)
$
specializing to $(1/q;1/q)_n=\prod_{i=1}^{n}\left(1-1/q^i\right)$:
\begin{equation}
  \label{ie_q_binomial_coefficient}
  1\le \frac{\gaussmnum{v}{k}{q}}{q^{k(v-k)}} \le \frac{1}{(1/q;1/q)_k}\le\frac{1}{(1/q;1/q)_\infty}\le \frac{1}{(1/2;1/2)_\infty}\approx 3.4627,
\end{equation}
where $(1/q;1/q)_\infty$ denotes $\lim_{n\to\infty} (1/q;1/q)_n$, cf.~the estimation for the Anticode bound in the proof of
Proposition~\ref{prop_optimal_johnson}. The sequence $(1/q;1/q)_\infty$
is monotonically increasing with $q$ and approaches $(q-1)/q$ for large $q$, see e.g.\ \cite{MR2451015} and \cite{khaleghi2009subspace}
for some numerical values.
\begin{lemma}
  \label{lemma_q_binomial_coefficient_limit}
  For each $b\in \mathbb{N}_{\ge 0}$ we have $\lim\limits_{a\to\infty} \frac{\gaussmnum{a+b}{b}{q}}{q^{ab}}=\frac{1}{(1/q;1/q)_b}$.
\end{lemma}
\begin{proof}
  Plugging in the definition of the $q$-binomial coefficient, we obtain
  \[
    \lim\limits_{a\to\infty} \frac{\gaussmnum{a+b}{b}{q}}{q^{ab}}=\lim\limits_{a\to\infty} \frac{\prod_{i=1}^b \frac{q^{a+i}-1}{q^i-1}}{q^{ab}}
    =\prod_{i=1}^b \frac{q^{i}}{q^i-1}=\prod_{i=1}^b \frac{1}{1-1/q^i}=\frac{1}{(1/q;1/q)_b}.
  \]
  
\end{proof}

Using this asymptotic result we can compare the size of the lifted MRD codes to the Singleton and the Anticode bound.

\begin{proposition}
  For $k \le v-k$ the ratio of the size of an LMRD code divided by the size of the Singleton bound converges
  for $v \rightarrow \infty$ monotonically decreasing to $(1/q;1/q)_{k-d/2+1}\ge (1/2;1/2)_\infty > 0.288788$.
\end{proposition}
\begin{proof}
Setting $z=k-d/2+1$ and $s=v-k$ the ratio is given by $g(s):=
\frac{q^{sz}}{\gaussmnum{s+z}{z}{q}}$, so that Lemma~\ref{lemma_q_binomial_coefficient_limit} gives the proposed limit.
The sequence is monotonically decreasing, since we have $0 \le z-1 < z \le s+z$ and
\[
\frac{g(s)}{g(s+1)}
=
\frac{q^{sz}\gaussmnum{s+1+z}{z}{q}}{\gaussmnum{s+z}{z}{q} q^{(s+1)z}}
=
\frac{q^{z}\gaussmnum{s+z}{z}{q} + \gaussmnum{s+z}{z-1}{q}}{\gaussmnum{s+z}{z}{q} q^{z}}
=
1+\frac{\gaussmnum{s+z}{z-1}{q}}{\gaussmnum{s+z}{z}{q} q^{z}}
>
1.
\]

\end{proof}

\begin{proposition}
  \label{prop_ratio_lmrd_anticode}
  For $k \le v-k$ the ratio of the size of an LMRD code divided by the size of the Anticode bound converges
  for $v \rightarrow \infty$ monotonically decreasing to $\frac{(1/q;1/q)_{k}}{(1/q;1/q)_{d/2-1}}\ge \frac{q}{q-1}\cdot (1/q;1/q)_{k}
  \ge 2\cdot (1/2;1/2)_\infty > 0.577576$.
\end{proposition}
\begin{proof}
  The LMRD code has cardinality $q^{(v-k)(k-d/2+1)}$ and the Anticode bound is $\gaussmnum{v}{k}{q} / \gaussmnum{v-k+d/2-1}{d/2-1}{q}$.
  From Lemma~\ref{lemma_q_binomial_coefficient_limit} we conclude
  \[
    \lim_{v\to\infty} \frac{\gaussmnum{(v-k)+k}{k}{q}}{q^{(v-k)k}}=\frac{1}{(1/q;1/q)_k}
    \text{ and }
    \lim_{v\to\infty} \frac{\gaussmnum{(v-k)+(d/2-1)}{d/2-1}{q}}{q^{(v-k)(d/2-1)}}=\frac{1}{(1/q;1/q)_{d/2-1}},
  \]
  so that the limit follows. The subsequent inequalities follow from $d\ge 4$, the monotonicity of $(1/q;1/q)_n$, and $q\ge 2$.

  It remains to show the monotonicity of the sequence
  \[
    g(v):=\frac{q^{(v-k)(k-d/2+1)}\gaussmnum{v-k+d/2-1}{d/2-1}{q}}{\gaussmnum{v}{k}{q}}.
  \]
  Using the abbreviation $s=v-k$ we define
  \[
    f(x):=\frac{\gaussmnum{s+x}{s+1}{q}}{q^x \gaussmnum{s+x}{s}{q}}
    =\frac{\prod_{i=1}^{s+1} \frac{q^{x-1+i}-1}{q^{i}-1}
    }{q^x\prod_{i=1}^{s} \frac{q^{x+i}-1}{q^{i}-1}}
    =
    \frac{\frac{q^{x}-1}{q^{s+1}-1}}{q^x}
    =
    \frac{1-q^{-x}}{q^{s+1}-1}
  \]
  and observe that $f$ is strictly monotonically increasing in $x$, so that $f(k) > f(d/2-1)$.
  Using routine manipulations of $q$-binomial coefficients we compute
  \begin{align*}
    \frac{g(v)}{g(v+1)}    =\left(1+f(k)\right)\cdot\left(1+f(d/2-1)\right)^{-1} >1.
  \end{align*}

\end{proof}

In other words the ratio between the best known lower bound and the best known upper bound for constant dimension codes is strictly greater
than $0.577576$ for all parameters and the most challenging parameters are given by $q=2$, $d=4$, and $k=\left\lfloor v/2\right\rfloor$.

Replacing the Anticode bound by the Johnson bound of Theorem~\ref{cor_johnson_opt} does not change the limit behavior of
Proposition~\ref{prop_ratio_lmrd_anticode} when $v$ tends to infinity. As stated above, we obtain the Anticode bound
if we remove the floors in Corollary~\ref{cor_johnson_opt} and replace $A_q(v-k+d/2,d;d/2)$ by $\frac{q^{v-k+d/2}-1}{q^{d/2}-1}$.
First we consider the latter weakening. Applying the lower bound of Theorem~\ref{thm:multicomponent} for $A_q(v',2k';k')$, where
$v'=tk'+r$ with $1 \le r \le k'-1$, we consider
\[
  \frac{q^{v'}-q^{k'+r}+q^{k'}-1}{q^{k'}-1} \,/\, \frac{q^{v'}-1}{q^{k'}-1}= 1-\frac{q^{k'}\cdot (q^r-1)}{q^{v'}-1}
\]
If $v'\ge 3k'$, then the subtrahend on the right hand side is at most $\frac{q^{v'/3}\cdot (q^{v'/3}-1)}{q^{v'}-1}$. Otherwise we have $2k'\le v'<3k'$,
so that $v'=2k'+r$. Since $q^{k'}\cdot (q^r-1)\cdot (q^{k'}+1)=q^{2k'+r}-q^{2k'}+q^{k'+r}-q^{k'}\le q^{2k'+r}-1$ the subtrahend on the right hand side
is at most $1/\left(q^{v'/3}+1\right)$. Thus, the ratio between the lower and the upper bound for partial spreads tends to
$1$ if $v'=v-k+d/2$ tends to infinity. Since
\begin{eqnarray*}
&&\left\lfloor \frac{q^{v}\!-\!1}{q^{k}\!-\!1} \left\lfloor \frac{q^{v\!-\!1}\!-\!1}{q^{k\!-\!1}\!-\!1} \left\lfloor \ldots
\left\lfloor \frac{q^{v\!-\!k\!+\!d/2\!+\!1}\!-\!1}{q^{d/2\!+\!1}\!-\!1} \left\lfloor\frac{q^{v-k+d/2}-1}{q^{d/2}-1}\right\rfloor \right\rfloor
\ldots \right\rfloor \right\rfloor \right\rfloor\\
&\ge& \left( \frac{q^{v}\!-\!1}{q^{k}\!-\!1} \left( \frac{q^{v\!-\!1}\!-\!1}{q^{k\!-\!1}\!-\!1} \left( \ldots
\left(\frac{q^{v-k+d/2}-1}{q^{d/2}-1}-1\right)\ldots \right)-1 \right)-1 \right)-1\\
&\ge& \frac{\gaussmnum{v}{k}{q}}{\gaussmnum{v-k+d/2-1}{d/2-1}{q}} -\left(k-d/2+1\right)\cdot \frac{\gaussmnum{v-1}{k-1}{q}}{\gaussmnum{v-k+d/2-1}{d/2-1}{q}}\\
&\ge& \frac{\gaussmnum{v}{k}{q}}{\gaussmnum{v-k+d/2-1}{d/2-1}{q}}\cdot\left(1-\frac{4\left(k-d/2+1\right)}{q^{v-k}}\right)
\end{eqnarray*}
the ratio between Corollary~\ref{cor_johnson_opt} and the Anticode bound tends to $1$ as $v$ tends to infinity.

Next, we consider the ratio between the lower bound from the first construction of Proposition~\ref{prop:nonrecursive_version}
and the Anticode bound when $l$ tends to infinity.

\begin{proposition}\label{prop:construction_1_anticode_bound_ratio}
  For integers satisfying the conditions of Proposition~\ref{prop:nonrecursive_version}, $k \le s$ and $d\le 2k$, we have
  \begin{eqnarray*}
    &&\lim_{l\to \infty }  \left(b(s)^l a(v_0)+a(s-d/2+k) \gaussmnum{l}{1}{b(s)}\right) /
    \frac{
    \gaussmnum{v_0+ls}{k}{q}
    }{
    \gaussmnum{v_0+ls-k+d/2-1}{d/2-1}{q}
    }\\
    &&=
    \frac{
    a(v_0)+ \frac{a(s-d/2+k)}{q^{s(k-d/2+1)}-1}
    }{
    q^{(v_0-k)(k-d/2+1)}
    }
    \cdot\prod\limits_{i=d/2}^k \left(1-\frac{1}{q^i}\right)
  \end{eqnarray*}
\end{proposition}
\begin{proof}
For $k \le s$ and $k-d/2+1\ge 1$ we have $b(s)\ne 1$, so that
\begin{align*}
b(s)^l a(v_0)+a(s') \gaussmnum{l}{1}{b(s)}
=
q^{lsk'} a(v_0)+a(s') \frac{q^{lsk'}-1}{q^{sk'}-1}
\\
=
q^{lsk'} \left(a(v_0)+ \frac{a(s')}{q^{sk'}-1} \right) - \frac{a(s')}{q^{sk'}-1},
\end{align*}
where we abbreviate $s'=s-d/2+k$ and $k'=k-d/2+1$. Thus,
\[
  \lim_{l\to\infty} \left(b(s)^l a(v_0)+a(s') \gaussmnum{l}{1}{b(s)}\right) / q^{lsk'}
  = a(v_0)+ \frac{a(s')}{q^{sk'}-1}.
\]
Plugging in the definition of the $q$-binomial coefficients gives
\[
  \frac{\gaussmnum{v_0+ls}{k}{q}}{\gaussmnum{v_0+ls-k+d/2-1}{d/2-1}{q}}
  =\frac{\prod\limits_{i=1}^k \frac{q^{v_0+ls-k+i}-1}{q^i-1}}{\prod\limits_{i=1}^{d/2-1}\frac{q^{v_0+ls-k+i}-1}{q^i-1}}
  = \prod\limits_{i=d/2}^k \frac{q^{v_0+ls-k+i}-1}{q^i-1},
\]
so that
\[
  \lim_{l\to\infty} \frac{\gaussmnum{v_0+ls}{k}{q}}{\gaussmnum{v_0+ls-k+d/2-1}{d/2-1}{q}} / q^{lsk'}
  = \prod\limits_{i=d/2}^k \frac{q^{v_0-k+i}}{q^i-1}
  = q^{(v_0-k)k'}\cdot\prod\limits_{i=d/2}^k \frac{1}{1-\frac{1}{q^i}}.
\]
Dividing both derived limits gives the proposed result.

\end{proof}

For Example~\ref{ex_ap_1} with $d=4$ and $k=3$, we obtain a ratio of
\[\left(1597245+ \frac{A_2(7,4;3)}{4095}\right) \cdot 21 /2^{25}
\in [0.99963386, 0.99963388]
\]
for $v=13+6l$ with $l\to\infty$ using $333 \le A_2(7,4;3) \le 381$, i.e., the Anticode bound is almost met
by the underlying improved linkage construction.

\section{Codes better than the MRD bound}
\label{sec_better_than_MRD_bound}

For constant dimension codes that contain a lifted MRD code, Theorem~\ref{theo:MRD_upper_bound} gives an upper bound
which is tighter than the Johnson bound of Theorem~\ref{thm_johnson_II}. In \cite{ai2016expurgation} two infinite series
of constructions have been given where the code sizes exceed the MRD bound of Theorem~\ref{theo:MRD_upper_bound}
for $q=2$, $d=4$, and $k=3$. Given the data available from \cite{HKKW2016Tables} we mention that, besides $d=4$, $k=3$, the only
other case where the MRD bound was superseded is $A_2(8,4;4)\ge 4801>4797$, see \cite{new_lower_bounds_cdc}. Next, we
show that for $d=4$ and $k=3$ the MRD bound can be superseded for all field sizes $q$ if $v$ is large enough. For the limit
of the achievable ratio we obtain:

\begin{proposition}
  \label{prop_better_than_mrdb_q}
  For $q\ge 3$ we have
  $
    \lim\limits_{v\to\infty} \frac{A_q(v,4;3)}{q^{2v-6}+\gaussmnum{v-3}{2}{q}}\ge 1+\frac{1}{2q^3}
  $.
\end{proposition}
\begin{proof}
  For $q\ge 2$, \cite[Theorem~4]{MR3444245} gives
  \[
    A_q(7,4;3) \ge q^8+q^5+q^4+q^2-q \ge q^8+q^5+q^4.
  \]
  With this, we conclude
  \[
    A_q(v_0,4;3)\ge A_q(7,4;3)\cdot q^{2v_0-14}+A_q(v_0-6,4;3)\ge q^{2v_0-10}\cdot \left(q^4+q+1\right)
  \]
  from Corollary~\ref{cor_improved_linkage} choosing $m=7$. Applying
  Proposition~\ref{prop:nonrecursive_version} with $s=3$ gives
  \[
    A_q(v_0+3l,4;3)\ge q^{6l} A_q(v_0,4;3)+\frac{q^{6l}-1}{q^6-1}\ge q^{6l} A_q(v_0,4;3)
  \]
  for $v_0\in \{12,13,14\}$, so that $A_q(v,4;3)\ge q^{2v-10}\cdot \left(q^4+q+1\right)$ for all $v\ge 12$.

  From Lemma~\ref{lemma_q_binomial_coefficient_limit} we conclude
  \[
    \lim_{v\to\infty} \frac{q^{2v-6}+\gaussmnum{v-3}{2}{q}}{q^{2v-10}}=q^4+(1/q;1/q)_2
    =\frac{q^3(q^4-q^3-q^2+q+1)}{(q-1)^2(q+1)}.
  \]
  Since
  \[
    \left(q^4+q+1\right) / \frac{q^3(q^4-q^3-q^2+q+1)}{(q-1)^2(q+1)} = 1+\frac{1}{q^3}-\frac{q+1}{q^2(q^4-q^3-q^2+q+1)},
  \]
  the statement follows for $q\ge 3$.
  
\end{proof}

For $q=2$ the estimations of the proof of Proposition~\ref{prop_better_than_mrdb_q} are too crude in order to obtain a factor larger
than one. However, for the binary case better codes with moderate dimensions of the ambient space have been found by computer searches --
with the prescription of automorphisms as the main ingredient in order to reduce the computational complexity, see e.g.\ \cite{MR2796712}.

\begin{proposition}
  \label{prop_better_than_mrdb_2}
For $v \ge 19$ we have
$\frac{A_2(v,4;3)}{2^{2v-6} + \gaussmnum{v-3}{2}{2}} \ge 1.3056$.
\end{proposition}
\begin{proof}
Applying Proposition~\ref{prop:nonrecursive_version} with $s=3$ and using $A_2(4,4;3)\ge 0$ gives
$A_2(v_0+3l,4;3) \ge A_2(v_0,4;3) \cdot 2^{6l}$
for all $v_0\ge 6$ and $l\ge 0$, so that
\begin{equation}
\label{ie_l_q}
\frac{A_2(v_0+3l,4;3)}{2^{2(v_0+3l)-6} + \gaussmnum{(v_0+3l)-3}{2}{2}}
\ge \frac{A_2(v_0,4;3)}{\frac{7}{3}\cdot 2^{2v_0-7}}.
\end{equation}
Using $A_2(7,4;3)\ge 333$ \cite{HKKW2016Tables},
$A_2(8,4;3)\ge 1326$ \cite{new_lower_bounds_cdc}, $A_2(9,4;3)\ge 5986$ \cite{new_lower_bounds_cdc}, and
$A_2(13,4;3)=1597245$ \cite{Braun16} we apply Corollary~\ref{cor_improved_linkage} with $m=13$ to obtain
lower bounds for $A_2(v_0,4;3)$ with $19 \le v_0 \le 21$. For these values of $v_0$ the minimum
of the right hand side of Inequality~(\ref{ie_l_q}) is attained at $v_0=20$ with value $1.3056442377$.
  
\end{proof}

Note that the application of Proposition~\ref{prop:nonrecursive_version} was used in a rather crude estimation in the proof
of Proposition~\ref{prop_better_than_mrdb_2}. Actually, we do not use the codewords generated by the codewords of cdc $C_2$ in
Theorem~\ref{theo:improved_linkage}, so that we might have applied \cite[Theorem~37]{silberstein2015error} directly for this part of the
proof -- similarly for Proposition~\ref{prop_better_than_mrdb_q}, which then allows to consider just one instead of $s=3$ \textit{starters}.
In the latter part of the proof of Proposition~\ref{prop_better_than_mrdb_2} the use of Corollary~\ref{cor_improved_linkage} is essential
in order to obtain large codes for medium sized dimensions of the ambient space from $A_2(13,4;3)=1597245$ and relatively good lower
bounds for small dimensions. This is a relative typical behavior of Corollary~\ref{cor_improved_linkage} and
Proposition~\ref{prop:nonrecursive_version}, i.e., the first few applications yield a significant improvement which quickly bottoms out
-- in a certain sense. As column \texttt{bklb} of Table~\ref{table_cmp_mrdb} suggests, we may slightly improve upon the value stated in
Proposition~\ref{prop_better_than_mrdb_2} by some fine-tuning effecting the omitted less significant digits.

\begin{table}[htp]
\begin{center}
\begin{tabular}{l|lll|ll|l}
$v$ & \texttt{bklb} & \texttt{mrdb} & \texttt{bkub} & \texttt{lold} & \texttt{lnew} & \texttt{ea} \\
\hline
6  &  77  &  71  &  77  &  65  &  65  &   \\
7  &  333  &  291  &  381  &  257  &  265  &  301 \\
8  &  1326  &  1179  &  1493  &  1033  &  1101  &  1117 \\
9  &  5986  &  4747  &  6205  &  4929  &  4929  &  4852 \\
10  &  23870  &  19051  &  24698  &  21313  &  21313  &  18924 \\
11  &  97526  &  76331  &  99718  &  85249  &  85257  &  79306 \\
12  &  385515  &  305579  &  398385  &  383105  &  383105  &  309667 \\
13  &  1597245  &  1222827  &  1597245  &  1532417  &  1532425  &  1287958 \\
14  &  6241665  &  4892331  &  6387029  &  6241665  &  6241665  &  4970117 \\
15  &  24966665  &  19571371  &  25562941  &  24966657  &  24966665  &  20560924 \\
16  &  102223681  &  78289579  &  102243962  &  102223681  &  102223681  &  79608330 \\
17  &  408894729  &  313166507  &  409035142  &  408894721  &  408894729  &   \\
18  &  1635578957  &  1252682411  &  1636109361  &  1635578889  &  1635578957  &   \\
19  &  6542315853  &  5010762411  &  6544674621  &  6542315597  &  6542315853  &
5200895489 \\
\end{tabular}
\caption{Lower and upper bounds for $A_2(v,4;3)$.}
\label{table_cmp_numbers}
\end{center}
\end{table}

In Tables~\ref{table_cmp_numbers}, \ref{table_cmp_lmrd}, and \ref{table_cmp_mrdb} we compare the sizes
of different constructions with the LMRD and the best known upper bound. Here \texttt{bklb} and \texttt{bkub} stand for
best known lower and upper bound respectively. The values of Theorem~\ref{theo:MRD_upper_bound} are given in column \texttt{mrdb}.
Applying Theorem~\ref{thm_original_linkage} and Theorem~\ref{theo:improved_linkage} to the best known codes give the columns
\texttt{lold} and \texttt{lnew}, respectively. The results obtained in \cite{ai2016expurgation} are stated in column \texttt{ea}.
The achieved ratio between the mentioned constructions and the MRD bound can be found in Table~\ref{table_cmp_mrdb}. Since
differences partially are beyond the given accuracy, we give absolute numbers in Table~\ref{table_cmp_numbers}. Note that the
values in column \texttt{bklb} of Table~\ref{table_cmp_mrdb} show that Proposition~\ref{prop_better_than_mrdb_2} is also valid
for $v\ge 16$, while we have a smaller ratio for $v<16$. The relative advantage
over lifted MRD codes is displayed in Table~\ref{table_cmp_lmrd}.

\begin{table}[htp]
\begin{center}
\begin{tabular}{l|lll|ll|l}
$v$ & \texttt{bklb} & \texttt{mrdb} & \texttt{bkub} & \texttt{lold} & \texttt{lnew} & \texttt{ea} \\
\hline
6  &  1.203125  &  1.109375  &  1.203125  &  1.015625  &  1.015625  &   \\
7  &  1.300781  &  1.136719  &  1.488281  &  1.003906  &  1.035156  &  1.175781 \\
8  &  1.294922  &  1.151367  &  1.458008  &  1.008789  &  1.075195  &  1.090820 \\
9  &  1.461426  &  1.158936  &  1.514893  &  1.203369  &  1.203369  &  1.184570 \\
10  &  1.456909  &  1.162781  &  1.507446  &  1.300842  &  1.300842  &  1.155029 \\
11  &  1.488129  &  1.164719  &  1.521576  &  1.300797  &  1.300919  &  1.210114 \\
12  &  1.470623  &  1.165691  &  1.519718  &  1.461430  &  1.461430  &  1.181286 \\
13  &  1.523252  &  1.166179  &  1.523252  &  1.461427  &  1.461434  &  1.228292 \\
14  &  1.488129  &  1.166423  &  1.522786  &  1.488129  &  1.488129  &  1.184968 \\
15  &  1.488129  &  1.166545  &  1.52367  &  1.488129  &  1.488129  &  1.225527 \\
16  &  1.523252  &  1.166606  &  1.523554  &  1.523252  &  1.523252  &  1.186257 \\
17  &  1.523252  &  1.166636  &  1.523775  &  1.523252  &  1.523252  &   \\
18  &  1.523252  &  1.166651  &  1.523746  &  1.523252  &  1.523252  &   \\
19  &  1.523252  &  1.166659  &  1.523801  &  1.523252  &  1.523252  &  1.210928 \\
\end{tabular}
\caption{Lower and upper bounds for $A_2(v,4;3)$ divided by the size of a corresponding lifted MRD code.}
\label{table_cmp_lmrd}
\end{center}
\end{table}

To conclude this section, we remark that an application of Corollary~\ref{cor_improved_linkage} with $2k \le m \le v-k$ using
a lifted MRD in the cdc $C_1$ cannot generate a code that exceeds the MRD bound of Theorem~\ref{theo:MRD_upper_bound}.

\begin{lemma}
Using the notation of Theorem~\ref{theo:improved_linkage}, let $k \le \min\{v_1-k,v_2-k+d/2\}$, $C_r$ a linear MRD code, $d_r=d_1/2$, and
$C_1$ contains a lifted MRD code (in $\gaussmset{\mathbb{F}_q^{v_1}}{k}$). Then, the codes constructed in
Theorem~\ref{theo:improved_linkage} contain a lifted MRD code (in $\gaussmset{\mathbb{F}_q^{v_1+v_2-k+d/2}}{k}$).
\end{lemma}
\begin{proof}
Let $\{\tau^{-1}(I_{k \times k} \mid M) : M \in R\} \subseteq C_1$ be the lifted MRD code in $C_1$. Since $R$ is a
$(k \times (v_1-k),d_1/2)_q$ MRD code, we have $\#R=q^{(v_1-k)(k-d_1/2+1)}$. The first set of the construction contains
\[\{\tau^{-1}(I_{k \times k} \mid M \mid A) : M \in R, A \in C_r\}\]
in which $\{(M \mid A) : M\in R, A \in C_r\}$ forms a $(k \times (v_1+v_2-2k+d/2),N,d_r)_q$ rank metric code of
size $N=q^{(v_1+v_2-2k+d/2)(k-d_r+1)}$, hence it is a maximum rank metric code.

\end{proof}

\begin{table}[htp]
\begin{center}
\begin{tabular}{l|lll|ll|l}
$v$ & \texttt{bklb} & \texttt{mrdb} & \texttt{bkub} & \texttt{lold} & \texttt{lnew} & \texttt{ea} \\
\hline
6  &  1.084507  &  1.0  &  1.084507  &  0.915493  &  0.915493  &   \\
7  &  1.144330  &  1.0  &  1.309278  &  0.883162  &  0.910653  &  1.034364 \\
8  &  1.124682  &  1.0  &  1.266327  &  0.876166  &  0.933842  &  0.947413 \\
9  &  1.261007  &  1.0  &  1.307141  &  1.038340  &  1.038340  &  1.022119 \\
10  &  1.252953  &  1.0  &  1.296415  &  1.118734  &  1.118734  &  0.993334 \\
11  &  1.277672  &  1.0  &  1.306389  &  1.116833  &  1.116938  &  1.038975 \\
12  &  1.261589  &  1.0  &  1.303705  &  1.253702  &  1.253702  &  1.013378 \\
13  &  1.306190  &  1.0  &  1.306190  &  1.253176  &  1.253182  &  1.053263 \\
14  &  1.275806  &  1.0  &  1.305519  &  1.275806  &  1.275806  &  1.015900 \\
15  &  1.275673  &  1.0  &  1.306140  &  1.275672  &  1.275673  &  1.050561 \\
16  &  1.305712  &  1.0  &  1.305972  &  1.305712  &  1.305712  &  1.016845 \\
17  &  1.305678  &  1.0  &  1.306127  &  1.305678  &  1.305678  &   \\
18  &  1.305661  &  1.0  &  1.306085  &  1.305661  &  1.305661  &   \\
19  &  1.305653  &  1.0  &  1.306124  &  1.305653  &  1.305653  &  1.037945 \\
\end{tabular}
\caption{Lower and upper bounds for $A_2(v,4;3)$ divided by the corresponding MRD bound.}
\label{table_cmp_mrdb}
\end{center}
\end{table}

\section{Conclusion}
\label{sec_conclusion}
In this paper we have considered the maximal sizes of constant dimension codes. With respect to constructive lower bounds we have
improved the so-called linkage construction, which then yields the best known codes for many parameters, see Footnote~\ref{footnote_linkage}.
With respect to upper bounds
there is a rather clear picture. The explicit Corollary~\ref{cor_johnson_opt}, which refers back to bounds for partial spreads, is the best
known parametric bound in the case of $d\neq 2\min\{k,v-k\}$, while Theorem~\ref{thm_ahlswede} or the linear programming method may
possibly yield improvements. Since Theorem~\ref{thm_ahlswede} implies the Johnson bound and so Corollary~\ref{cor_johnson_opt}, it would
be worthwhile to study whether it can be strictly sharper than Theorem~\ref{thm_johnson_II} for $d\neq 2\min\{k,v-k\}$ at all. Compared to
Corollary~\ref{cor_johnson_opt}, the only two known improvements are given for the specific parameters from Theorem~\ref{thm_specific_bound_1}
and Proposition~\ref{prop_specific_bound_2}. In the case of partial spreads we have reported the current state-of-the-art
mentioning that further improvements are far from being unlikely.

In general we have shown that the ratio between the best-known lower and upper bound is strictly larger than $0.577576$ for all
parameters. The bottleneck is formed by the parameters $q=2$, $d=4$, and $k=\left\lfloor v/2\right\rfloor$, where no known method can properly
improve that factor, see Footnote~\ref{footnote_linkage} for the linkage construction. For $d=4$, $k=3$ and general field sizes $q$ we have applied the improved linkage construction in order to show
that $A_q(v,d;k)$ is by a factor, depending on $q$, larger than the MRD bound for sufficiently large dimensions $v$.

\section*{Acknowledgement}
The authors would like to thank
Harout Aydinian for providing an enlarged proof of Theorem~\ref{thm_ahlswede},
Natalia Silberstein for explaining the restriction $3k \le v$ in \cite[Corollary~39]{silberstein2015error},
Heide Gluesing-Luerssen for clarifying the independent origin of the linkage construction,
and Alfred Wassermann for discussions about the asymptotic results of Frankl and R{\"o}dl.


\end{document}